 \documentclass[11pt]{article}
 \usepackage[top=1in,bottom=1in,left=1.5in,right=1.5in]{geometry}
  \usepackage{amsmath,amssymb}
  \usepackage{latexsym}
 \usepackage[dvips]{pstricks} 
  \usepackage{pst-node}
  \usepackage[dvips]{graphicx}

 \newcommand\R{\mathord{\mathbb R}}
 
 \newcommand\C{\mathord{\mathbb C}}

 \newcommand\F{\mathord{\mathbb F}}

 \newcommand\N{\mathord{\mathbb N}}
 
 \renewcommand\P{\mathord{\mathbb P}}

  \renewcommand{\i}{\mathbf{i}}

  \renewcommand{\r}{\mathbf{r}}
  \newcommand{\e}{\mathbf{e}}
  \newcommand{\f}{\mathbf{f}}

  \newcommand{\n}{\mathbf{n}\mathnormal}

  \renewcommand{\u}{\mathbf{u}}

  \newcommand{\w}{\mathbf{w}}
  \newcommand{\x}{\mathbf{x}}
  
  \newcommand{\y}{\mathbf{y}}
  
  \newcommand{\z}{\mathbf{z}}
  \newcommand{\0}{\mathbf{0}}

 \renewcommand\f{{\bf f}\mathnormal}

 \newcommand\bF{{\bf F}}

 \newcommand\bH{\mathbf {H}}

 \newcommand\cS{{\cal S}}
 \newcommand\cT{{\cal T}}

 \newcommand\cY{{\cal Y}}

 \newcommand\rS{{\rm S}}

  \newcommand{\lan}{\langle}
  \newcommand{\ran}{\rangle}
  \newcommand{\an}[1]{\lan#1\ran}
  \def\diag{\mathop{{\rm diag}}\nolimits}
  \newcommand{\hs}{\hspace*{\parindent}}

  \newcommand{\trans}{^\top}
  \newcommand{\qed}{\hspace*{\fill} $\Box$\\}

  \newcommand{\dist}{\mathrm{dist}}

  \renewcommand{\rS}{\mathrm{S}}

  \newtheorem{theo}{\bfseries \hs Theorem}[section]
  \newtheorem{defn}[theo]{\bfseries \hs Definition}

  \newtheorem{lemma}[theo]{\bfseries \hs Lemma}
  \newtheorem{corol}[theo]{\bfseries \hs Corollary}
  
  \newtheorem{con}[theo]{\bfseries \hs Conjecture}

  \numberwithin{equation}{section} 

 \newtheorem{remark}[theo]{\bfseries \hs Remark}
 \newtheorem{theorem}{Theorem}[section]
 \newtheorem{exm}[theorem]{Example}

 \renewcommand\det{{\rm det\;}}
 \renewcommand\dim{{\rm dim\;}}

 \renewcommand\deg{{\rm deg\;}}

\begin{document}

 \title{The geometric measure of entanglement of symmetric $d$-qubits is computable}
 \author{
  Shmuel Friedland and Li Wang\footnotemark[1]
 }
 \renewcommand{\thefootnote}{\fnsymbol{footnote}}

 \footnotetext[1]{
 Department of Mathematics, Statistics and Computer Science,
 University of Illinois at Chicago, Chicago, Illinois 60607-7045,
 USA, \texttt{friedlan@uic.edu,liwang8@uic.edu}.    
 }

 \renewcommand{\thefootnote}{\arabic{footnote}}
 \date{March 4, 2017}
 \maketitle
 \begin{abstract}  We give a simple formula for finding the spectral norm of a $d$-mode symmetric tensor in two variables
 	over the complex or real numbers
 	in terms of the complex or real roots of a corresponding polynomial in one complex variable.  This result implies that the geometric  measure
 	of  entanglement of symmetric $d$-qubits is polynomial-time computable.  We discuss a generalization to  $d$-mode symmetric tensors in more than two variables.

 \end{abstract}


\section{Introduction}
The spectral norm of a matrix has numerous applications in pure and applied mathematics.
One of the fundamental reasons for the tremendous use of this norm is that it is polynomial-time computable and the software for its computation is easily available 
on MAPLE, MATHEMATICA, MATLAB and other platforms.

Multiarrays, or $d$-mode tensors, are starting to gain popularity due to data explosion and other applications.  Usually, these problems deal with tensors with real numbers.
Since the creation of quantum mechanics, $d$-mode tensors over complex numbers became the basic tool in treating the $d$-partite states.  Furthermore,
the special case of $d$-partite qubits, viewed as $\otimes^d\C^2$, the tensor product of $d$ copies of $\C^2$, is the basic ingredient in building the quantum computer.  

The spectral norm of tensors is a well defined quantity for tensors over the real or complex numbers denoted as $\R$ and $\C$ respectively.  (In this paper we let $\F$ be either $\R$ or $\C$.) 
For complex valued tensors $\cT$ of Hilbert-Schmidt norm one, the spectral norm measures the geometric measure of entanglement of $\cT$, the most important feature in quantum information theory.
(See \S\ref{sec: spwcnrm}.)
Unlike in the matrix case, the computation of the spectral norm in general can be NP-hard \cite{FLTN16,HL13}.  However, there is a need to compute these norms in special cases
of interesting applications.  Even the simplest case of $d$-partite qubits poses  theoretical and numerical challenges \cite{GFE09}.  This can be partly explained by the fact that 
the space $\otimes^d\C^2$ has dimension $2^d$.

In this paper we mostly restrict ourselves  to $d$-symmetric tensors over $\F^n$, denoted as $\rS^d\F^n$.  The dimension of this space is ${n+d- 1\choose d}={n+d-1\choose n-1}$.  Hence for fixed $n$ this dimension is $O(d^{n-1})$.  In particular, the dimension of $\rS^d\C^2$, the space of symmetric $d$-qubits, is $d+1$.
A symmetric tensor $\cS\in\rS^d\F^n$ can be identified with a homogeneous
polynomial of degree $d$ in $n$ variables over $\F$.  It was already observed by J. J. Sylvester \cite{Syl51} that binary forms, i.e., $n=2$, posses very special properties related to polynomials of one complex variable.  

Chen, Xu and Zhu observed in \cite{CXZ10} that the spectral norm of a symmetric $d$-qubit with nonnegative entries can be computed by finding the corresponding maximal real root
of a certain real polynomial.  In \cite{AMM10} the authors computed numerically the spectral norm of symmetric qubits using Majorana representation combined with analytical and numerical results.  

The main purpose of this paper is to give an analytic expression for the spectral norm of a $d$-symmetric qubit, i.e. $\cS\in\rS^d\C^2$, in terms of the roots of the corresponding
polynomial of one complex variable of degree at most $(d-1)^2+1$, provided that this symmetric qubit is not in the exceptional family.  For the exceptional family of $d$-symmetric qubits, we give a polynomial time approximation algorithm.   If  $\cS$ is real valued then its real spectral norm depends only on the real roots of this polynomial,
or actually, on the real root of another polynomial of degree at most $d+1$.  

Recall that the problem of finding all complex valued roots of univariate polynomials with precision $\varepsilon$ is polynomial-time computable \cite{NR96}.
In particular, we deduce that the geometric measure of entanglement of symmetric $d$-qubits is polynomial-time computable.

In principle we can extend these results to tensors in $\rS^d\C^n$ for $n>2$.  This will require solving a system of polynomial equations in $n$ complex variables, which is a much harder task \cite{BCSS98}.  We conjecture in \S\ref{sec:nge3} that for a fixed $n\ge 3$, for most symmetric $d$-qudits, the spectral norm of $\cS\in\rS^d\C^n$ is polynomial-time computable in $d\in\N$. 

We now survey briefly the contents of our paper. In \S\ref{sec: spwcnrm} we state our notations for tensors.  We recall the definition of the spectral norm of a tensor
$\cT$. We state the well known connection between the notion of the geometric measure of entanglement and the spectral norm of the $d$-partite state. 
In \S\ref{sec:specnrmsym} we discuss the spectral norm of $d$-symmetric tensors on $\F^n$.  We consider a standard orthonormal basis in $\rS^d\C^n$,  the analog of Dicke states in $\rS^d\C^2$ \cite{Dic}, and the entanglement of each element in the basis. 
We give an upper bound on the entanglement of symmetric states in $\rS^d\C^n$.
In \S\ref{sec:critpts}
we recall the remarkable theorem of Banach \cite{Ban38} that characterizes the spectral norm of a symmetric tensor, which was rediscovered a number of times in the mathematical and physical literature \cite{CHLZ12,Fri13,Hubetall09}.  Let $f(\x)$ be a homogeneous polynomial of degree $d$, whose maximum and minimum on the corresponding unit sphere in $\F^n$ gives the spectral norm of the corresponding symmetric tensor.
We show that the critical points of the real part of $\f(\x)$ are anti-fixed and fixed points of the corresponding polynomial maps in $\F^n$.  Using the degree theory we give lower and upper bounds on the number of complex anti-fixed points for nonsingular $\cS\in\rS^d\C^n$.  (The set of singular $\cS\in\rS^d\C^n$  is a variety \cite{FO14}.)  \S\ref{sec:dqubit} is the most important section of this paper.  Here we give a formula to compute the spectral norm of a symmetric tensor $\cS\in\rS^d\F^2$.
This formula depends on the complex or real zeros of a corresponding polynomial of degree at most $(d-1)^2+1$.  In particular, we give a formula
to compute the geometric measure of entanglement of symmetric $d$-qubits states.  Unfortunately, our formula is not applicable for a special one real parameter family.
The analysis of this family and the computation of the spectral norm of the corresponding symmetric tensors $\cS\in\rS^d\C^2$ within $\varepsilon$ precision is given in \S\ref{sec:excepcase}. 
In \S\ref{sec:nge3} we discuss briefly the complexity of finding the spectral norm of $\cS\in\rS^d\C^n$ for a fixed $n\ge 3$ and arbitrary $d$.  We give some evidence for our conjecture that for a nonsingular $\cS$ the computation of the spectral norm of $\cS$ within precision $\varepsilon>0$ is polynomial in $d$.
Appendix \ref{sec:examples} gives numerical examples of our method for calculating the spectral norm of $\cS\in\rS^d\F^2$.


\section{Spectral norm and entanglement}\label{sec: spwcnrm}
For a positive integer $d$, i.e., $d\in\N$, we denote by $[d]$ the set of consecutive integers $\{1,\ldots,d\}$.
Let $\F\in\{\R,\C\}$, $\n=(n_1,\ldots,n_d)\in\N^d$.  We will identify the tensor product space
$\otimes_{i=1}^d \F^{n_i}$ with the space of $d$-arrays $\F^{\mathbf{n}}$.
The entries of $\cT\in\F^{\mathbf{n}}$ 
are denoted as $\cT_{i_1,\ldots,i_d}$.
So $d=1,d=2$ and $d\ge 3$ are called vectors, matrices and tensors respectively.   Note that the dimension of $\F^{\mathbf{n}}$ is $N(\n)=n_1\cdots n_d$. 

Assume that $d\in\N$ and $k\in [d]$.
Suppose that $\mathbf{m}=(m_1,\ldots,m_k)$ is obtained by deleting some coordinates of $\mathbf{n}$.  That is, $\mathbf{m}=(n_{l_1},\ldots,n_{l_k})$, where
$1\le l_1<\ldots<l_k\le d$.   Denote by $\mathbf{m}'\in \N^{d-k}$ the vector obtained from $\mathbf{n}$ by deleting the coordinates $l_1,\ldots,l_k$.
That is $\mathbf{m}'=(n_{l_1'},\ldots,n_{l_{d-k}'})$, where $1\le l_1'<\cdots<l_{d-k}'\le d$.  (It is possible that $d-k=0$.)
Denote by $\cS\times \cT=\cT\times \cS$ the multiarray in $\F^{\mathbf{m}'}$ obtained by the contraction on the indices $i_{l_1},\ldots,i_{l_k}$:
\[(\cT\times \cS)_{i_{l_1'},\ldots,i_{l'_{d-k}}}=\sum_{i_{l_1}\in[n_{l_1}],\ldots,i_{l_k}\in[n_{l_k}]} \cT_{i_1,\ldots,i_d}\cS_{i_{l_1},\ldots,i_{l_d}}.\]
The inner product on $\F^{\mathbf{n}}$ is given as $\an{\cS,\cT}:=\cS\times \overline{\cT}$, where $\overline{\cT}=[\overline{\cT_{i_1,\ldots,i_d}}]$.
Furthermore, $\|\cS\|=\sqrt{\an{\cS,\cS}}$ is the Hilbert-Schmidt norm of $\cS$.  Assume that $\x_i=(x_{1,i},\ldots,x_{n_i,i})\trans\in\F^{n_i}$ for $i\in[d]$.
Then $\otimes_{i=1}^d \x_i$ is a tensor in $\F^{\mathbf{n}}$, with the entries $(\otimes_{i=1}^d \x_i)_{i_1,\ldots,i_d}=x_{i_1,1}\cdots x_{i_d,d}$.
($\otimes_{i=1}^d \x_i$ is called a rank one tensor if all $\x_i\ne \0$.)

Denote the unit sphere in $\F^n$ by $\rS(n,\F)=\{\x\in\F^n, \|\x\|=1\}$. Recall that the spectral norm of $\cT\in \F^{\mathbf{n}}$ is given as
\begin{equation}\label{specnrmdef}
\|\cT\|_{\sigma,\F}=\max\{|\cT\times \otimes_{i=1}^d \x_i|, \; \x_i\in \rS(n_i,\F) \textrm{ for } i\in[d]\}.
\end{equation}
Unlike in the matrix case, for a real tensor $\cT\in \R^{\mathbf{n}}$ it is possible that $\|\cT\|_{\sigma,\R}<\|\cT\|_{\sigma,\C}$ \cite{FLTN16}.
For simplicity of notation we will let $\|\cT\|_{\sigma}$ denote $\|\cT\|_{\sigma,\C}$, and no ambiguity will arise.

A  standard way to compute the spectral norm of $\cT$ is an alternating maximization in (\ref{specnrmdef}) by maximizing each time with respect to
a different variable \cite{LMV00}.  Other variants of this method is maximization on two variables using the SVD algorithms \cite{FMPS13}, or the
Newton method \cite{FT15,Tong}.  These methods in the best case yield a convergence to a local maximum, which provide a lower bound
to $\|\cT\|_{\sigma,\F}$.  Semidefinite relaxation methods, as in \cite{Nie14}, will yield an upper bound to $\|\cT\|_{\sigma,\F}$, which will converge in some cases to $\|\cT\|_{\sigma,\F}$.

Recall that in quantum physics $\cT\in\C^{\mathbf{n}}$ is called a state if $\|\cT\|=1$.  Furthermore, all tensors of the form $\zeta\cT$, where $\|\cT\|=1$ and 
$\zeta\in\C,|\zeta|=1$ are viewed as the same state.  That is, the space of the states in $\C^{\mathbf{n}}$ is the quotient space $\rS(N(\n),\C)/\rS(1,\C)$.  Denote by $\Pi^{\mathbf{n}}$ the product states in $\C^{\mathbf{n}}$:
\[\Pi^{\mathbf{n}}=\{\otimes_{i=1}^d \x_i,\; \x_i\in\rS(n_i,\C),i\in[d]\}.\]

The geometric measure of entanglement of a state $\cT\in\C^{\mathbf{n}}$ is 
\[\dist(\cT,\Pi^{\mathbf{n}})=\min_{\cY\in\Pi^{\mathbf{n}}} \|\cT-\cY\|.\]
As $\|\cT\|=\|\cY\|=1$ it follows that $\dist(\cT,\Pi^{\mathbf{n}})=\sqrt{2(1-\|\cT\|_{\sigma})}$.  Hence an equivalent measurement of entanglement is \cite{GFE09}
\begin{equation}\label{defetaT}
\eta(\cT)=-\log_2 \|\cT\|_\sigma^2.
\end{equation}
The maximal entanglement is
\begin{equation}\label{maxentn}
\eta(\n)=\max_{\cT\in\C^{\mathbf{n}}, \|\cT\|=1} -\log_2\|\cT\|_\sigma^2.
\end{equation}
See \cite{DFL} for other measurements of entanglement using the nuclear norm of $\cT$.  Lemma 9.1 in \cite{FLTN16} implies
\[\eta(\n)\le \log_2 N(\n).\]

Let $n^{\times d}=(n,\ldots,n)\in\N^d$.
For $n=2$ we get that $\eta(2^{\times d})\le d$.  In \cite{Jungetall08} it is shown that $\eta(2^{\times d})\le d-1$.
On the other hand, it is shown in \cite{GFE09} that $\eta(\cT)\ge d - 2\log_2 d-2$
for the set of states of Haar measure at least $1-e^{-d^2}$ on the sphere $\|\cT\|=1$ in $\otimes^d\C^{2}$.

A tensor $\cS=[\cS_{i_1,\ldots,i_d}]\in\otimes^d\F^n$ is called symmetric if $\cS_{i_1,\ldots,i_d}=
\cS_{i_{\omega(1)},\ldots,i_{\omega(d)}}$ for every permutation $\omega:[d]\to[d]$.  
Denote by $\rS^d\F^n\subset \otimes^d\F^n$ the vector space of $d$-mode symmetric tensors on $\F^n$.
In what follows we assume that $\cS$ is a symmetric tensor and $d\ge 2$, unless stated otherwise.  A tensor
$\cS\in\rS^d\F^n$ defines a unique homogeneous polynomial of degree $d$ in $n$ variables
\begin{equation}\label{defpolfx}
f(\x)=\cS\times\otimes^d\x=\sum_{0\le j_k\le d,k\in[n], j_1+\cdots +j_n=d} \frac{d!}{j_1!\cdots j_n!} f_{j_1,\ldots,j_n} x_1^{j_1}\cdots x_n^{j_n}.
\end{equation}
Conversely, a homogeneous polynomial $f(\x)$ of degree $d$ in $n$ variables defines a unique symmetric $\cS\in\rS^d\F^n$ by the following relation.
Consider the multiset $\{i_1,\ldots,i_d\}$, where each $i_l\in [n]$.  Let $j_k$ be the number of times the integer $k\in [n]$ appears in the multiset  $\{i_1,\ldots,i_d\}$.
Then $\cS_{i_1,\ldots,i_d}=f_{j_1,\ldots,j_n}$.  Furthermore
\begin{equation}\label{symtenhsnorm}
\|\cS\|^2=\sum_{0\le j_k\le d,k\in[n],j_1+\cdots + j_n=d}\frac{d!}{j_1!\cdots j_n!} |f_{j_1,\ldots,j_n}|^2,
\end{equation}
where $\cS_{i_1,\ldots,i_d}=f_{j_1,\ldots,j_n}$.  
\section{Standard basis of symmetric tensors and their entanglement}\label{sec:specnrmsym}
\begin{defn}\label{defJdnDIc}

\item Denote by $J(d,n)$ be the set of all $n$-tuples $\{j_1,\ldots,j_n\}$ appearing in \eqref{defpolfx}:
	\begin{equation}\label{defJdn}
	J(d,n)= \{\{j_1,\ldots,j_n\},\; j_k\in\{0,1,\ldots,d\} \textrm{ for } k\in [n], j_1+\cdots+j_n=d\}.
	\end{equation}
\begin{enumerate}
\item For each $\{j_1,\ldots,j_n\}\in J(d,n)$ let $\cS(j_1,\ldots,j_n)\in\rS^d\C^n$ be the following symmetric tensor with entries $\cS_{i_1,\ldots,i_d}$ for $i_1,\ldots,i_d\in[n]$:
		The entry $\cS_{i_1,\ldots,i_d}=\sqrt{\frac{j_1!\cdots j_n!}{d!}}$ if $k$ appears 
		$j_k$ times in the multiset $\{i_1,\ldots,i_d\}$ for each $k\in[n]$.   Otherwise $\cS_{i_1,\ldots,i_d}=0$. 
\item  Let 
		\begin{eqnarray}\label{defSdn}
		&&\cS(d,n)=\cS(j_1,\ldots,j_n), \textrm{ where}\\ 
		&&j_1=\cdots =j_l=\lfloor\frac{d}{n}\rfloor,\;j_{l+1}=\cdots = j_{n}=\lceil\frac{d}{n}\rceil, \;l=n\lceil\frac{d}{n}\rceil - d.\label{defj1jn}
		\end{eqnarray}
\end{enumerate}		
\end{defn}	
In the following lemma we show that the set of the above vectors $\cS(j_1,\ldots,j_n)$ is an orthonormal basis for $\rS^d\F^n$. 
We call this basis a \emph{standard} basis of symmetric tensors.	
For $n=2$ the standard basis of symmetric tensors is called  the Dicke basis \cite{Dic}. 
\begin{lemma}\label{lem:Sdn} Assume that $n,d\ge 2$ are two positive integers.  
	\begin{enumerate}
		\item Denote by $|J(d,n)|$ the cardinality of the set $J(d,n)$.  Then $|J(d,n)|={n+d-1\choose n-1}$.  Furthermore, $\dim\rS^d\F^n= |J(d,n)|$.
	\item The set $\cS(j_1,\ldots,j_n), \{j_1,\ldots,j_d\}\in J(d,n)$ is an orthonormal basis for $\rS^d\F^n$. 		
\end{enumerate}
\end{lemma}
\begin{proof} \emph{1}.   Consider the homogeneous function $f(\x)$ given by (\ref{defpolfx}).
	Each monomial in $f(\x)$ corresponds exactly to the $n$-tuple $\{j_1,\ldots,j_n\}\in J(d,n)$.
	As pointed out at the beginning of this section each $\{j_1,\ldots,j_n\}$ corresponds to a unique multiset $\{i_1,\ldots,i_d\}$ in the product set $[n]^{\times d}$.  Without loss of generality we can assume that $1\le i_1\le \cdots \le i_d\le n$.  Let $i_p'=i_p+p-1$ for $p\in[d]$.  Then $\{i_1',\ldots,i_d'\}$  is a set of $d$ distinct integers in $[n+d-1]$.  
	The number of such sets is ${n+d-1\choose d}={n+d-1\choose n-1}$.  Hence $|J(d,n)|={n+d-1\choose n-1}$.  Clearly $\dim \rS^d\F^n={n+d-1\choose n-1}$.
	
	\noindent 
	\emph{2}.  The equality (\ref{symtenhsnorm}) yields that $\cS(j_1,\ldots,j_n)$ has norm one.  Clearly the set of symmetric tensors  $\cS(j_1,\ldots,j_n), \{j_1,\ldots,j_n\}\in J(d,n)$ is a basis in $\rS^d\F^n$.  Let $\cT\in\rS^d\F^n$ and assume that
	\[\cT\times \otimes^d\x=\sum_{j_k+1\in[d+1],k\in[n], j_1+\cdots+j_k=d}\frac{d!}{j_1!\cdots j_n!} g_{j_1,\ldots,j_n}x_1^{j_1}\cdots x_n^{j_n}.\]
	Then
	\[\an{\cS,\cT}=\sum_{j_k+1\in[d+1],k\in[n], j_1+\cdots+j_k=d}\frac{d!}{j_1!\cdots j_n!}  f_{j_1,\ldots,j_n}\overline{g_{j_1,\ldots,j_n}}.\]
	Hence $\cS(j_1,\ldots,j_n): \{j_1,\ldots,j_n\}\in J(d,n)$ is an orthonormal basis for $\rS^d\F^n$.\qed
	\end{proof}

In the following lemma we find the entanglement of each $\cS(j_1,\ldots,j_n)$ and the maximum entanglement of these states.
\begin{lemma}\label{lem:Sdn1} Assume that $n,d\ge 2$ are two positive integers.  Then
	\begin{enumerate}
	\item  For each $\{j_1,\ldots,j_n\}\in J(d,n)$ the following equality holds
		\begin{equation}\label{entanfSj1jn}
		\eta(\cS(j_1,\ldots,j_n))=\log_2d^d - \log_2 d! + \sum_{k=1}^n (\log_2 j_k!-\log_2 j_k^{j_k}).
		\end{equation}
		\item 
		\begin{eqnarray}\label{specnrmSdn}
		&&\|\cS(j_1,\ldots,j_n)\|_{\sigma}\ge \|\cS(d,n)\|_{\sigma}=\sqrt{ \frac{ d!\big(\lfloor\frac{d}{n}\rfloor\big)^{l\lfloor\frac{d}{n}\rfloor}\big(\lceil\frac{d}{n}\rceil\big)^{(n-l)\lceil\frac{d}{n}\rceil}}{d^d \big(\lfloor\frac{d}{n}\rfloor !\big)^l\big(\lceil\frac{d}{n}\rceil !\big)^{n-l}}},\\
		\label{etaSdn}
		&&\eta(\cS(j_1,\ldots,j_n))\le \eta(\cS(d,n))=\log_2 \frac{d^d \big(\lfloor\frac{d}{n}\rfloor !\big)^l\big(\lceil\frac{d}{n}\rceil !\big)^{n-l}}{ d!\big(\lfloor\frac{d}{n}\rfloor\big)^{l\lfloor\frac{d}{n}\rfloor}\big(\lceil\frac{d}{n}\rceil\big)^{(n-l)\lceil\frac{d}{n}\rceil}},
		\end{eqnarray}
		for each  $\{j_1,\ldots,j_n\}\in J(d,n)$.
		\item Assume that the integer $n\ge 2$ is fixed and $d\gg 1$.  Then
		\begin{equation}\label{asforetaSdn}
		\eta\big(\cS(d,n)\big)=\frac{1}{2}\big((n-1)\log_2 d - n\log_2 n\big)+O(\frac{1}{d}).
		\end{equation}
	\end{enumerate}
\end{lemma}
\begin{proof} \emph{1}.
	Clearly
	\[|\cS(j_1,\ldots,j_n)\times\otimes^d\x|^2=\frac{d!}{j_1!\cdots j_n!}(|x_1|^2)^{j_1}\cdots (|x_n|^2)^{j_n}.\]
	Use Lagrange multipliers to deduce that the maximum of the above function for $\|\x\|=1$ is achieved at the points $|x_k|^2=\frac{j_k}{j_1+\cdots+j_n}=\frac{j_k}{d}$ for $k\in[n]$.  Hence
	\begin{equation}\label{specnrmSj1jn}
	\|\cS(j_1,\ldots,j_n)\|_{\sigma,\R}^2=\|\cS(j_1,\ldots,j_n)\|_{\sigma}^2=\frac{d!\prod_{k=1}^n j_k^{j_k}}{d^d\prod_{k=1}^n j_k!}.
	\end{equation}
	This establishes (\ref{entanfSj1jn}).
	
	\noindent
	\emph{2}.  Let $a,b$ be nonnegative integers such that $a\le b-2$.  We claim that 
	\[\frac{a!b!}{a^a b^b}< \frac{(a+1)!(b-1)!}{(a+1)^{a+1}(b-1)^{b-1}}.\]
	Indeed, the above inequality is equivalent to 
	\[\frac{a!(a+1)^{a+1}}{a^a(a+1)!}<\frac{(b-1)!b^b}{(b-1)^{b-1}b!}\iff \left(\frac{a+1}{a}\right)^a <\left(\frac{b}{b-1}\right)^{b-1}.\]
	As $0^0=1$ we deduce that the above inequalities hold for $a=0$ and $b\ge 2$.  Assume that $a\ge 1$.
	Then the last inequality in the above displayed inequality is equivalent to the well known statement that the sequence $(1+\frac{1}{m})^m$ is a strictly increasing .  
	
	Consider $\|\cS(j_1,\ldots,j_n)\|^{-2}$.  Suppose that there exists $j_p,j_q$ such that $|j_p-j_q|\ge 2$.  Without loss of generality we may assume that $j_p\le j_q-2$.
	Let $j'_l=j_l$ for $l\in [n]\setminus\{p,q\}$, and $j'_p=j_p+1, j'_q=j_q-1$.  Then the above inequality yields that $\|\cS(j_1,\ldots,j_n)\|^{-2}<\|\cS(j_1',\ldots,j_n')\|^{-2}$.
	Hence the maximum value of $\|\cS(j_1,\ldots,j_n)\|^{-2}$, where $\{j_1,\ldots,j_n\}\in J(d,n)$, is achieved for $\{j_1,\ldots,j_n\}$ satisfying $|j_p-j_q|\le 1$ for all $p,q\in[n]$.  Without loss of generality we can assume that $j_1,\ldots,j_n$ are given by (\ref{defj1jn}).  Hence (\ref{etaSdn}) holds.
	
	\noindent
	\emph{3}.  The equality (\ref{asforetaSdn}) follows from Sterling's formula \cite[p. 52]{Fel58}
	\[k!=\sqrt{2\pi k} k^k e^{-k}e^{\theta_k/12 k}, \quad 0<\theta_k<1.\] 
\qed	
\end{proof}

We now comments on the results given in Lemma \ref{lem:Sdn}.  Parts \emph{1} and \emph{2} are well known.  For $n=2$ Lemma \ref{lem:Sdn} is well known in physics community \cite{AMM10}.  The states $\cS(j_1,j_2)$ are called \emph{Dicke} states.
Note that
\[
\|\cS(3,2)\|_{\sigma}=\frac{2}{3}\approx 0.6667,\;\|\cS(4,2)\|_{\sigma}=\frac{\sqrt{6}}{4}\approx 0.61237,\;\|\cS(5,2)\|_{\sigma}=\frac{6\sqrt{6}}{25}\approx 0.5879.\]
It is known that the most entangled $3$-qubit state with respect to geometric measure is $\cS(3,2)$ \cite{TWP09,CXZ10}.  That is, the spectral norm of a nonsymmetric $3$-qubit
is not less than the spectral norm of $\cS(3,2)$, which is equivalent to the equality
$\eta((2,2,2))=\eta(\cS(3,2))$, see (\ref{maxentn}). However for $d>3$, Lemma \ref{lem:Sdn} shows that the states $\cS(d,2)$ are not the most entangled states in $\rS^d\C^2$.  See examples in \cite{AMM10}, which are also discussed in Appendix \ref{sec:examples}.

Lemma 4.3.1 in \cite{Ren05} yields that
\begin{equation}\label{upbndentsym}
\eta(\cS)\le \log_2 {n+d-1\choose n-1}, \quad \cS\in\rS^d\C^n, \|\cS\|=1.
\end{equation}
(See also \cite{MGBB10}.)   In particular, for $n=2$ we have the inequality:
\begin{equation}\label{qubitsentangi}
\eta(\cS)\le \log_2(d+1) \quad \cS\in\rS^d\C^2, \|\cS\|=1.
\end{equation}

Note that for a fixed $n$ and large $d$ we have the complexity expression
\begin{equation}\label{asymptfornchos}
\log_2 {n+d-1\choose n-1}= (n-1) \log_2 (d+1) +\frac{(n-1)(n-2)}{2\ln 2} - \log_2 (n-1)!+O\left(\frac{1}{d}\right).
\end{equation}

Let 
\begin{equation}\label{defetadn}
\eta_{sym}(d,n)=\max\{\eta(\cS),\; \cS\in\rS^d\C^n, \|\cS\|=1\}.
\end{equation}
Combining the inequality (\ref{upbndentsym})  with \eqref{etaSdn} we obtain
\begin{equation}\label{lowupbdsetadn}
\log_2 \frac{d^d \big(\lfloor\frac{d}{n}\rfloor !\big)^l\big(\lceil\frac{d}{n}\rceil !\big)^{n-l}}{ d!\big(\lfloor\frac{d}{n}\rfloor\big)^{l\lfloor\frac{d}{n}\rfloor}\big(\lceil\frac{d}{n}\rceil\big)^{(n-l)\lceil\frac{d}{n}\rceil}}\le \eta_{sym}(d,n)\le  \log_2 {n+d-1\choose n-1}, \;l=n\lceil\frac{d}{n}\rceil -d.
\end{equation}
In particular
\begin{equation}\label{lowupbdsetad2}
\log_2\frac{d^d\big(\lfloor\frac{d}{2}\rfloor!\big)\big(\lceil\frac{d}{2}\rceil!\big)}{d! \big(\lfloor\frac{d}{2}\rfloor\big)^{\lfloor\frac{d}{2}\rfloor} \big(\lceil\frac{d}{2}\rceil\big)^{\lceil\frac{d}{2}\rceil}}\le \eta_{sym}(d,2)\le \log_2(d+1).
\end{equation}

There is a gap of factor $2$ between the lower and  the upper bounds in  (\ref{lowupbdsetadn}) and (\ref{lowupbdsetad2}) for fixed $n$ and $d\gg 1$.
In \cite{FK16} it is shown that the following inequality holds with respect to the corresponding Haar measure on the unit ball $\|\cS\|=1$ in $\rS^d\C^2$:
\begin{equation}\label{coplexentangsym1}
\Pr(\eta(\cS)\le \log_2 d -\log_2 (\log _2 d) +\log_2 4 -\log_2 5)\le \frac{1}{d^6}.
\end{equation}
This shows that the upper bounds in (\ref{lowupbdsetad2}) have the correct order.  In particular, (\ref{coplexentangsym1}) is the analog of the inequality $\eta(\cT)\ge d - 2\log_2 d-2$
for most $d$-qubit states in \cite{GFE09}. 

We now define a relative entanglement of a symmetric state $\cS\in \rS^d\C^n$, denoted as $\eta_{rel}(\cS)$. The value of $\eta_{rel}(\cS)$  for small values of $d$ should give an idea how entangled is $\cS$, independently of the value of $d$.  Let
\begin{equation}\label{defrelentnagl}
\eta_{rel}(\cS)=-\log_2\|\cS\|_{\sigma}^2 -\log_2 {n+d-1\choose n-1}, \quad \cS\in\rS^d\C^n, \|\cS\|=1.
\end{equation}
In particular, 
\begin{equation}\label{defrelentnagl2}
\eta_{rel}(\cS)=-\log_2\|\cS\|_{\sigma}^2 -\log_2 {(d+1)}, \quad \cS\in\rS^d\C^2, \|\cS\|=1.
\end{equation}
Thus $\eta_{rel}(\cS)\le 0$. The inequality (\ref{coplexentangsym1}) shows that $\eta_{rel}(\cS)\ge -\log_2(\log_2 d)$ for most of $\cS$ for $d\gg 1$ and $n=2$.  
For large values of $d$ we can't rule out that $-\eta_{rel}(\cS)$ can be quite large for most of $\cS$.  

\begin{exm}\label{table:entangl} 
	In this example we give the table of $\eta(\cS)$ and $\eta_{rel}(\cS)$ for $d$ from $3$ to $12$ for the most entangled $\cS\in \rS^d\C^2$, which are known to us. For $d=3$ we let $\cS=\cS(3,2)$.  For $d$ from $4$ to $12$ we choose the states from \cite[\S6]{AMM10}.
	\begin{table}\caption{Computational results for Example \ref{table:entangl}.}
		\label{table:exm:entangl}
		\centering
		\begin{scriptsize}
			\begin{tabular}{|c|l|l||c|l|l|} \hline
				$d$ &    $\eta(\cS)$ & $\eta_{rel}(\cS)$ & $d$ &  $\eta(\cS)$     & $\eta_{rel}(\cS)$  \\ 
				\hline 
				3& 1.1699  &  -0.8301 & 8 & 2.45    &-0.7199    \\ \hline              
				4 & 1.5850     & -0.7370  & 9 & 2.554  & -0.7679   \\  \hline   
				5 & 1.74     & -0.8450  &10& 2.7374& -0.7220      \\  \hline 
				6 & 2.1699   &  -0.6374 &11&  2.83   & -0.7550 \\  \hline 
				7 & 2.299   &  -0.701   &12&  3.1175& -0.5829     \\  \hline   
			\end{tabular} 
		\end{scriptsize}
	\end{table} 
\end{exm} 

This table shows that the relative entanglement of the most entangled symmetric states known to us for $d$ from $3$ to $12$ has relatively small variation,
to compare with the variation of the entanglement of these states.
\section{Critical points of $\Re (\cS\times\otimes^d \x)$ on $\rS(n,\F)$ }\label{sec:critpts}
Recall that $\cS\in\rS^d\C^n$ is called nonsingular \cite{FO14} if 
\[\cS\times\otimes^{d-1}\x=\0\Rightarrow \x=\0.\]
Otherwise $\cS$ is called singular.
A nonzero homogeneous polynomial $f(\x)$ defines a hypersurface $H(f):=\{\x\in\C^{n}\setminus\{\0\},f(\x)=0\}$ in the $n-1$ projective space $\P\C^n$.  
$H(f)$ is called a smooth hypersurface if $\nabla f(\x)\ne \0$ for each $\x\ne \0$ that satisfies $f(\x)=0$.
The following lemma is probably well known to the experts.
\begin{lemma}\label{relatfxS}  Assume that $\cS\in\rS^d\C^n$.
	Let $f(\x)=\cS\times\otimes^d\x$.  Then
	\begin{eqnarray}\label{Fformulas}
	&&\bF(\x)=\cS\times \otimes^{d-1}\x=\frac{1}{d}\nabla f(\x)=\frac{1}{d}(\frac{\partial f}{\partial x_1}(\x),\ldots,\frac{\partial f}{\partial x_n}(\x)),\\
	\label{Eulerform}
	&&\sum_{i=1}^d x_iF_i(\x)=f(\x)=\cS\times\otimes^d\x.
	\end{eqnarray}
	$\cS$ is nonsingular if and only if $H(f)$ is a smooth hypersurface in $\P\C^n$.
\end{lemma}
\begin{proof} Assume that $\cS=[\cS_{i_1,\ldots,i_d}]\in \rS^d\F^n$.  Then 
	\[f(\x)=\sum_{i_j\in[n], j\in [d]}\cS_{i_1,\ldots,i_d}x_{i_1}\cdots x_{i_d}.\]
	Hence
	\[\frac{\partial f}{\partial x_l}=\sum_{k=1}^d\left(\sum_{i_k=l, i_j\in [n], j\in[d]\setminus\{k\}} \cS_{i_1,\ldots,i_d} \frac{x_{i_1}\cdots x_{i_d}}{x_{l}}\right).\]
	As $\cS$ is symmetric we can interchange in the above formula $i_k=l$ with $i_1$ for each $k\in [d]$.  Hence $d F_l = \frac{\partial f}{\partial x_l}$.  This proves (\ref{Fformulas}). The equality (\ref{Eulerform}) is Euler's identity.
	The equality (\ref{Fformulas}) implies that $\cS$ is nonsingular if and only if $H(f)$ is a smooth hypersurface.\qed 
\end{proof}

The remarkable result of Banach \cite{Ban38} claims that the spectral norm of a symmetric
tensor can be computed as a maximum on the set of rank one symmetric tensors:
\begin{equation}\label{Banthm}
\|\cS\|_{\sigma,\F}=\max\{|\cS\times \otimes^d \x|, \; \x\in\rS(n,\F), \textrm{ for } \cS\in\rS^d\F^n \}.
\end{equation}
This result was rediscovered several times since 1938.  In quantum information theory (QIT), for the case $\F=\C$, it appeared in \cite{Hubetall09}.  In mathematical literature, for the case $\F=\R$, it appeared in \cite{CHLZ12,Fri13}.  (Observe that  a natural generalization of Banach's theorem to partially symmetric tensors is given in \cite{Fri13}.)

Fix $\x\in\C^n$ and let $\zeta\in \C$.  Then $\cS\times\otimes^d\left(\zeta\x\right)=\zeta^d\left(\cS\times\otimes^d\x\right)$.  Hence there exists $\zeta\in\C, |\zeta|=1$ such that
$|\cS\times\otimes^d\x|=\Re\left(\cS\times\otimes^d\left(\zeta\x\right)\right)$.
Therefore for  $\F=\C$ we can replace the characterization (\ref{specnrmdef}) with:
\begin{equation}\label{maxcharspecnrmc}
\|\cS\|_{\sigma}=\max_{\x\in\rS(n,\C)} \Re(\cS\times\otimes^d\x), \textrm{ for } \cS\in\rS^d\C^n.
\end{equation}
\begin{lemma}\label{critptlem}  Assume that $\cS\in\rS^d\F^n, d\ge 2$.  
	A point $\x\in\rS(n,\F)$ is a critical point of $\Re f(\x)$ on $\rS(n,\F)$
	if and only if
	\begin{equation}\label{critpt}  
	\cS\times \otimes^{d-1}\x=\lambda\overline{\x}, \quad \x\in\rS(n,\F), \lambda\in\R,
	\end{equation}
	where $\overline{\x}$ denote the complex conjugate of $\x$.   The number of critical values $\lambda$ satisfying (\ref{critpt}) is finite.
\end{lemma}
\begin{proof} 
	Assume first that $\F=\R$.  Let $\x\in\rS(n,\R)$.  Suppose first that $\x$ is a critical point of $f(\z)=\cS\times\otimes^d\z$ for $\z\in\rS(n,\R)$.  Let $\y\in\R^n$ be orthogonal to $\x$: $\y\trans \x=0$.  Then $\|\x+t\y\|=\sqrt{1+t^2\|\y\|^2}=1+O(t^2)$ for $t\in\R$.  Clearly
	\[\cS\times \otimes ^d (\x+t\y)=\cS\times \otimes^d \x+td\y\trans (\cS\times \otimes^{d-1}\x) +O(t^2).\]
	As $\x$ is a critical point of $\cS\times\otimes^d\z$ for $\z\in\rS(n,\R)$ it follows that $y\trans (\cS\times \otimes^{d-1}\x)=0$ for each $\y$ orthogonal to $\x$.
	Hence $\cS\times \otimes^{d-1}\x$ is colinear with $\x$.  As $\bar\x=\x$ for each $\x\in\R^n$ we deduce (\ref{critpt}).  Similar arguments show that if (\ref{critpt}) holds for
	$\x\in\rS(n,\R)$ then $\x$ is a critical point.  
	
	As $f(\x)$ is a polynomial on $\R^n$ it follows that 
	its restriction on $\rS(n,\R)$ has a finite number of critical points \cite{Mil}.  This proves \emph{3} for $\F=\R$.
	
	Assume second that $\F=\C$.  View $\C^{n}$ as $2n$-dimensional real vector space $\R^{2n}$ with the standard inner product $\Re (\y^*\x)$, where $\y^*=\overline{\y}\trans$.  Hence $\|\x\|=\sqrt{\Re (\x^*\x)}$.  Assume that $\x\in\rS(n,\C)$
	is a critical point of $\Re (\cS\times \otimes^d\z)$ on $\rS(n,\C)$.  Let $\y\in\C^n$ be orthogonal to $\x$: $\Re (\y^*\x)=\Re(\overline{ \y}\trans \x)=0$.  Then $\|\x+t\y\|=\sqrt{1+t^2\|\y\|^2}
	=1+O(t^2)$ for $t\in\R$.   Hence
	\[\Re(\cS\times \otimes ^d (\x+t\y))=\Re(\cS\times \otimes^d \x)+td\Re(\y\trans (\cS\times \otimes^{d-1}\x)) +O(t^2).\]
	As $\x$ is a critical point we deduce that 
	\[0=\Re(\y\trans (\cS\times \otimes^{d-1}\x))=\Re(\y^* (\overline{\cS\times \otimes^{d-1}\x})).\]
	Hence $\overline{\cS\times \otimes^{d-1}\x}$ is $\R$-colinear with $\x$.  Thus (\ref{critpt}) holds. 
	As $q(\x):=\Re(\cS\times \otimes^d\x)$ is a polynomial on $\C^n\sim\R^{2n}$ it follows that 
	its restriction on $\rS(n,\C)$ has a finite number of critical points \cite{Mil}.  This proves \emph{3} for $\F=\C$.\qed
\end{proof}

Clearly, a maximum point of $|\cS\times \otimes^d\x|$ on $\rS(n,\F)$ is a critical point of $\Re\left(\cS\times\otimes^d\z\right)$ on $\rS(n,\F)$.  Hence
\begin{corol}\label{maxeig}  Let the assumptions and results of Lemma \ref{critptlem} hold.  Then
\begin{enumerate} 
		\item Assume that $\cS\in\rS^d\R^n$.  Then there exists $\x\in\rS(n,\R)$ satisfying (\ref{critpt}) such that $|\lambda|=\|\cS\|_{\sigma,\R}$.  Furthermore, $\|\cS\|_{\sigma,\R}$ is the maximum of all $|\lambda|$ satisfying (\ref{critpt}).
		\item Assume that $\cS\in\rS^d\C^n$.  Then there exists $\x\in\rS(n,\C)$ satisfying (\ref{critpt}) such that $\lambda=\|\cS\|_{\sigma}$.  Furthermore, $\|\cS\|_{\sigma}$ is the maximum of all $|\lambda|$ satisfying (\ref{critpt}).
\end{enumerate}
\end{corol}

We call $\x\in\rS(n,\F)$ and $\lambda\in\F$ an eigenvector and an eigenvalue of $\cS\in\rS^d\F^n$ if the following conditions hold \cite{CS}:
\begin{equation}\label{defeigvl}
\cS\times\otimes^{d-1}\x=\lambda\x, \quad \x\in\rS(n,\F),\;\lambda\in \F, \quad \cS\in\rS^d\F^n.
\end{equation}

Assume that $\F=\R$.  Then (\ref{defeigvl}) is equivalent to (\ref{critpt}).  Assume first that $d$ is odd and $\x$ is an eigenvector of $\cS$.  Then $-\x$ is an eigenvector of $\cS$ corresponding to $-\lambda$.  Hence without loss of generality we can consider only nonnegative eigenvalues of (\ref{defeigvl}).  Assume second that $d$ is even and $\x$ is an eigenvector of $\cS$.  Then $-\x$ is also eigenvector of $\cS$ corresponding to $\lambda$.

Suppose that $\F=\C$.  Assume that $\x\in\rS(n,\C)$ and $\lambda\in\C$ are an eigenvector and the corresponding eigenvalue of $\cS\in\rS^d\C^n$.  Let $\zeta\in\C,|\zeta|=1$.
Then $\zeta\x$ is an eigenvector of $\cS$ with the corresponding eigenvalue $\zeta^{d-2}\lambda$.  Assume that $\lambda\ne 0$.  For $d>2$ we can choose $\zeta,|\zeta|=1$ such that $\zeta^{d-2}\lambda=|\lambda|>0$.  
Furthermore,  the number of such choices of $\zeta$ is $d-2$.   In this context it is natural to consider the eigenspace $\mathrm{span}(\x)$, to which correspond a unique eigenvector
$\lambda\ge 0$.
It is shown in \cite{CS} that the number of different 
eigenspaces of generic $\cS\in\rS^d\C^n$ is
\begin{equation}\label{CSnumbereigv}
c(2,n) =n, \quad c(d,n)=\frac{(d-1)^n -1}{d-2} \textrm{ for } d\ge 3.
\end{equation}
Hence for generic $\cS\in\rS^d\R^n$ one has the above number of eigenspaces $\mathrm{span}(\x), \x\in\rS(n,\C)$.  The obvious question is what is the maximal number
of eigenspaces $\mathrm{span}(\x)$ corresponding to $\x\in\rS(n,\R)$ for generic $\cS\in\rS^d\R^n$.  Since $\cS\times \otimes^d\x$ has at least two critical points on
$\rS(n,\R)$ for $\cS\ne 0$ it follows that $\cS\ne 0$ has at least one real eigenspace. 

A vector $\x\in\rS(n,\C)$ and a scalar $\lambda\in\R$ that satisfy (\ref{critpt}) are called the \emph{anti-eigenvector} and \emph{anti-eigenvalue} of $\cS\in\rS^d\C^n$.
Note that if $\x$ is an anti-eigenvector and $\lambda$ a corresponding anti-eigenvalue then $\zeta\x$ is also anti-eigenvector with a corresponding anti-eigenvalue $\varepsilon\lambda$,  where $\varepsilon=\pm 1$ and $\zeta^d=\varepsilon$.  Hence, we can always assume that each nonzero anti-eigenvalue is positive, and there are $d$ different choices of $\zeta$ such that $\zeta\x\in\mathrm{span}(\x)$ is an anti-eigenvector corresponding to a given positive anti-eigenvalue $\lambda$.

The above result for symmetric complex valued matrices is a particular case of Schur's theorem.  Namely, assume that $T\in\rS^2\C$, i.e. $T$ is a complex symmetric matrix.
Then there exists a unitary matrix $U\in \C^{n\times n}$ such that $U\trans AU=\diag(a_1,\ldots,a_n), a_1\ge \cdots\ge a_n\ge 0 $. Here $a_i=\sigma_i(T), i\in [n]$ are the singular values of $T$.    
Let $U=[\u_1,\cdots,\u_n] $.  Then $AU=\bar U \diag(a_1,\ldots,a_n)$ which is equivalent to $A\u_i=a_i \bar\u_i, i\in[n]$, which is a special case of (\ref{critpt}).  

We now give an estimate of the number of different positive anti-eigenvalues for a generic $\cS\in\rS^d\C^n$.

\begin{theorem}\label{estnumbanteig}  Assume that $\cS\in\rS^d\C^n$ is nonsingular.  Then the number of positive anti-eigenvalues with corresponding anti-eigenspaces
	is finite.  This number $\mu(\cS)$, counting with multiplicities,  satisfies the inequalities 
	\begin{equation}\label{estnumbanteig1}
	\frac{(d-1)^n-1}{d}\le \mu(\cS)\le \frac{(d-1)^{2n} -1}{(d-1)^2-1}.
	\end{equation}
\end{theorem}

\begin{proof} Assume that $\cS\in\rS^d\C^n$ is nonsingular.  Suppose first that $d=2$. Schur's theorem implies that the number of different positive anti-eigenvalues of a complex symmetric 
	matrix, which are the singular values of $\cS$, is at most $n$.  Hence  (\ref{estnumbanteig1}) holds.
	
	Suppose second that $d>2$.
	Assume that $\x\in\rS(n,\C)$ is an anti-eigenvector with corresponding anti-eigenvalue $\lambda> 0$.  Choose $\y=(\lambda)^{-\frac{1}{d-2}}\x$ to obtain:
	\begin{equation}\label{antifixpoint}
	\cS\times \otimes^{d-1}\y=\bF(\y)=\bar \y.
	\end{equation}
	
	Consider the system (\ref{antifixpoint}).  It can be viewed as a system of $2n$ polynomial equations with $2n$ real variables when we identify $\C^n$ with $\R^{2n}$.
	Then it has a trivial solution $\y=0$ with multiplicity one since the linear term of $\mathbf{G}(\y):=\cS\times\otimes^{d-1}\y-\bar\y$ is $-\bar\y$.  Hence the Jacobian of $\mathbf{G}(\x)$ at $\y=\0$
	is invertible.
	Let $\y\ne \0$ be a solution of (\ref{antifixpoint}).  Then $\x=\frac{1}{\|\y\|}\y$ is an anti-eigenvector corresponding to the anti-eigenvalue $\|\y\|^{-\frac{1}{d-2}}$.
	
	We now show that (\ref{antifixpoint}) has a finite number of solutions.  Denote 
	\begin{equation}\label{defFbarF}
	\bar{\bF}:\C^n\to\C^n,\bar{\bF}(\y):=\bar{\cS}\times \otimes^{d-1}\y.
	\end{equation}
	Hence (\ref{antifixpoint}) is equivalent to  $\y=\overline{\bF(\y)}=\bar{\bF}(\bar\y)$.
	Let 
	\begin{equation}\label{defH}
	\bH: \;\C^n\to \C^n, \quad \bH=\bar{\bF}\circ\bF.
	\end{equation}
	Observe first that $\cS$ is nonsingular if and only if $\bar\cS$ is nonsingular.  The assumption that $\cS$ is nonsingular yields 
	\begin{equation}\label{nonsingHmap}
	\bH(\y)=\bar\bF(\bF(\y))=\0 \Rightarrow \bF(\y)=\0 \Rightarrow \y=\0.
	\end{equation}
	Let $\bH(\y)=(H_1(\y),\ldots,H_n(\y))\trans$.  Then each $H_i(\y)$ is a nonzero homogeneous polynomial of degree $(d-1)^2$.
	
	Observe next that each $\y\in\C^n$ that satisfies (\ref{antifixpoint}) is a fixed point of $\mathbf{H}$:
	\begin{equation}\label{fixpointH}
	\bH(\y)=\y.
	\end{equation}
	Let $\mathbf{K}(\y):=\bH(\y)-\y$.   As the principle homogeneous part of $\mathbf{K}$ is $\mathbf{H}$ it follows 
	that the map $\mathbf{K}:\C^n\to \C^n$ is a proper map \cite{Fri77}, i.e. $\lim_{\|\y\|\to\infty}\|\mathbf{K}(\y)\|=\infty$.
	Hence $\mathbf{K}$ is a branched cover of $\C^n$ of degree $\deg(\mathbf{K})=(d-1)^{2n}$.  In particular, $\mathbf{K}^{-1}(\0)$
	consists of at most $(d-1)^{2n}$ distinct points.  If we count these points with multiplicities then their number is exactly $(d-1)^{2n}$ .
	Clearly, $\mathbf{K}(\0)=\0$.  As the Jacobian of $\mathbf{K}$ at $\0$ is nonsingular it follows that $\0$ is a simple solution of $\mathbf{K}(\y)=\0$.
	Hence the number of nonzero points in $\mathbf{K}^{-1}(\0)$ is exactly $(d-1)^{2n}-1$, if we count each nonzero point with its multiplicity.
	As we explained above if $\mathbf{K}(\y)=\0, \y\ne \0$ then  $\mathbf{K}(\zeta\y)=0$ for each $\zeta$ satisfying $\zeta^{(d-1)^2-1}=1$.This observation
	yields the second inequality in (\ref{estnumbanteig1}).
	
	We now prove the first inequality in (\ref{estnumbanteig1}) using the degree theory as in \cite{Fri77}.  For $t\in\R$ let
	$\mathbf{G}_t(\y)=\bF(\y)-t\bar\y$.  As the principle homogeneous part of $\mathbf{G}_t$ is $\mathbf{F}$ for each $t\in\R$ it follows that $\mathbf{G}_t:\C^n\to \C^n$
	is a proper map. View the $2n$-dimensional sphere $\rS^{2n}\subset \R^{2n+1}$ as the one point compactification of $\C^n$: $\rS^{2n-1}\sim \C^n\cup\{\infty\}$.
	Extend $\mathbf{G}_t$ to the map $\widehat {\mathbf{G}}_t: \C^n\cup\{\infty\}\to\C^n\cup\{\infty\}$ by letting $\widehat {\mathbf{G}}_t(\infty)=\infty$.
	As $\mathbf{G}_t$ is proper it follows that $\widehat{\mathbf{G}}_t$ is continuous on  $\C^n\cup\{\infty\}$.  Hence we can define the topological degree of the map 
	$\widehat{\mathbf{G}}_t$ denoted $\deg \widehat{\mathbf{G}}_t$.   It is straightforward to show that the map $\widehat{\mathbf{G}}_t$ is continuous in the parameter
	$t$.  Hence $\deg \widehat{\mathbf{G}}_t$ does not depend on $t$.  In particular  $\deg \widehat{\mathbf{G}}_t= \deg \widehat{\mathbf{G}}_0= \deg \widehat{\mathbf{F}}$.
	$\mathbf{F}$ is a proper polynomial map in $n$-complex variables on $\C^n$.  Hence its degree is $(d-1)^n>0$.  Therefore $\deg \widehat{\mathbf{G}}_1= \deg \widehat{\mathbf{G}}=(d-1)^n$.  As $\mathbf{G}$ is a real polynomial map in $2n$ real variables it follows that $\mathbf{G}^{-1}(\w)$ is a finite set $\{z_1(\w),\ldots,\z_{N(\w)}(\w)\}$, for most of the points $\w\in\C^n$,
	where the Jacobian of $\mathbf{G}$ is invertible.  Let $\varepsilon(\z_i(\w))\in \{-1,1\}$ be the sign of the determinant of the Jacobian of $\mathbf{G}$ at $\z_i(\w)$, viewed as a real matrix of order $2n$.   Then
	\[(d-1)^n=\deg\widehat{\mathbf{G}}=\sum_{i=1}^{N(\w)} \varepsilon(\z_i(\w)).\]
	Therefore the number of preimages of most of $\w$ is at least $(d-1)^n$.  Recall that we showed that the set $\mathbf{G}^{-1}(\0)$ is a finite set.
	Hence counting with multiplicities, i.e. the minimum number of preimages of $\mathbf{G}^{-1}(\w)$ for small $\|\w\|$, we deduce that this number is at least
	$(d-1)^n$.  Recall that $\y=0$ is a simple root of $\mathbf{G}(\y)=\0$.  Hence the number of nonzero roots of $\mathbf{G}(\y)=\0$, counted with their multiplicities is at
	least $(d-1)^n-1$.  Each nonzero root $\y$ gives rise to $d$ distinct solutions $\zeta\y$, where $\zeta^d=1$.  These arguments give the lower bound in  (\ref{estnumbanteig1}).\qed
\end{proof} 

\begin{remark}\label{remark1} In \cite{HS14} the authors consider the dynamics of a special anti-holomorphic map of $\C$ of the form $z\mapsto \bar z^d +c$.  They also note that the dynamics of the ``squared" map is given by the holomorphic map $z\mapsto (z^d +\bar c)^d +c$.
\end{remark}
Thus the dynamics of the maps $\bF$ and $\bH$ are generalizations of the dynamics studied in \cite{HS14}.

In what follows we will need the following observation:
\begin{lemma}\label{FHest} Assume that $\cS\in\rS^d\F^n$.  Let $\bF$ and $\mathbf{H}$ be defined as in (\ref{defFbarF}) and (\ref{defH}) respectively.
	Then
	\begin{equation}\label{FHest1}
	\|\bF(\y)\|\le \|\cS\|_{\sigma,\F}\|\y\|^{d-1}, \quad \|\bH(\y)\|\le \|\cS\|_{\sigma,\F}^d \|\y\|^{(d-1)^2}.
	\end{equation}
	For $\y\in\rS(n,\F)$ satisfying $|\cS\times \otimes^d\y|=\|\cS\|_{\sigma,\F}$ equality holds in the above inequalities.  Suppose furthermore that $d>2$ and $\y\ne \0$ is a fixed point of $\bH$.  Then
	\begin{equation}\label{FHest2}
	\|\y\|^{-(d-2)}\le \|\cS\|_{\sigma,\F}.
	\end{equation}
\end{lemma}
\begin{proof}  Since $\bF$ and $\mathbf{H}$ are homogeneous maps of degree $d-1$ and $(d-1)^2$ respectively, it is enough to prove the inequalities (\ref{FHest1}) for $\y\in\rS(n,\F)$.
	Assume that $\y\in\rS(n,\F)$.  Let $\w=\cS\times \otimes^{d-1}\y$.  Assume first that $\w=\0$.  Then $\bF(\y)=\bH(\y)=\0$ and  (\ref{FHest1}) trivially holds.
	Assume second that $\w\ne \0$.  Let $\z=\frac{1}{\|\w\|}\overline {\w}$.  Hence 
	\[\|\bF(\y)\|=|\cS\times (\z\otimes (\otimes^{d-1}\y))|\le \|\cS\|_{\sigma,\F}.\]
	This establishes the first inequality in (\ref{FHest1}).  Clearly, $\|\bar\cS\|_{\sigma,\F}=\|\cS\|_{\sigma,\F}$.  Hence
	\[\|\bH(\y)\|=\|\bar\bF(\bF(\y))\|\le \|\cS\|_{\sigma,\F}(\|\bF(\y)\|)^{d-1}\le  \|\cS\|_{\sigma,\F} ( \|\cS\|_{\sigma,\F})^{d-1}= \|\cS\|_{\sigma,\F}^d.\]
	This establishes the second inequality in (\ref{FHest1}).
	
	Suppose that  $|\cS\times \otimes^d\y|=\|\cS\|_{\sigma,\F}$ for $\y\in\rS(n,\F)$.
	Assume first that $\F=\C$.
	Hence there exists $\zeta\in\C,|\zeta|=1$ such that $\x=\zeta\y$ satisfies (\ref{critpt}) with $\lambda=\|\cS\|_{\sigma}$.  Clearly $\|\bF(\x)\|=\lambda=\|\cS\|_{\sigma}$.
	Moreover 
	\[\bH(\x)=\bar \bF(\lambda\bar\x)=\lambda^{d-1}\bar \bF(\bar\x)=\lambda^d\x=
	\|\cS\|^{d}_{\sigma}\x.\]
	Hence $\|\bH(\x)\|=\|\cS\|^{d}_{\sigma}$.  Since $\bF$ and $\bH$ are homogeneous it follows that $\|\bF(\y)\|=\|\cS\|_{\sigma}$ and $\|\bH(\y)\|=\|\cS\|_{\sigma}^d$.
	
	Assume second that $\F=\R$.  Then $\y\in\rS(n,\R)$ is a critical point of $\cS\times \otimes^d\x$ on $\rS(n,\R)$.  Corollary \ref{maxeig} yields that $\cS\times \otimes^{d-1}\y=\pm \|\cS\|_{\sigma,\R}\y$.  Hence $\|\bF(\y)\|=\|\cS\|_{\sigma,\R}$ and $\|\bH(\y)\|=\|\cS\|_{\sigma,\R}^d$.
	
	Assume finally that $\bH(\y)=\y, \y\ne \0$.  The second inequality of (\ref{FHest1}) yields $\|\y\|= \|\bH(\y)\|\le  \|\cS\|^d_{\sigma,\F}\|\y\|^{(d-1)^2}$.  Hence
	$ \|\cS\|^d_{\sigma,\F}\ge \|\y\|^{-(d-1)^2+1}=\|\y\|^{-d(d-2)}$ which implies (\ref{FHest2}).\qed
\end{proof}
\section{Polynomial-time computability of spectral norm of symmetric $d$-qubits}\label{sec:dqubit}
In this section we slightly change our notations as follows.  First, the vectors in $\F^2$ are denoted as $\x=(x_0,x_1)\trans$.  Second, as common in physics, we assume that the 
$\cS\in \rS^d\F^2$ has entries $s_{i_1,\ldots,i_d}$ where $i_1,\ldots,i_d\in\{0,1\}$. 
(That is, $s_{i_1,\ldots,i_d}=\cS_{i_1+1,\ldots, i_d+1}$.)
Thus $\cS$ is parametrized by a vector $\mathbf{s}=(s_0,\ldots,s_d)$ where $s_{i_1,\ldots,i_d}=s_k$ if exactly $k$ indices from the multiset
$\{i_1,\ldots,i_d\}$ are equal to $1$.  Note that exactly $d \choose k$ entries of $\cS$ are equal to $s_k$.
Hence
\begin{equation}\label{normalization}
\|\cS\|=\sqrt{\sum_{k=0}^d {d\choose k}|s_k|^2}.
\end{equation} 
Recall that a qubit state is identified with the class of all vectors 
\[\{\zeta\x: \x=(x_0,x_1)\trans\in\C^2, \|\x\|=1, \zeta\in\C,|\zeta|=1\}.\] 
Denote by $\Gamma$ the set of all qubits.  Then $\Gamma$ can be identified with the Riemann sphere $\C\cup\{\infty\}$.   Indeed associate with a qubit $\x=(x_0,x_1)\trans, x_0\ne 0$ a unique complex number $z=\frac{x_1}{x_0}\in\C$.  The qubit $\x=(0,x_1), |x_1|=1$ corresponds to $z=\infty$.

The following lemma is a preparation result to state the main theorem of this section.
\begin{lemma}\label{qubspecnrmlem}  Let $\cS\in\rS^d\C^2$ and associate with $\cS$ the vector $\mathbf{s}=(s_0,\ldots,s_d)\trans\in \C^{d+1}$.
	Then
	\begin{enumerate}
		\item
		Let $f(\x)=\cS\times\otimes ^d\x$ and  $\cS\times\otimes^{d-1} \x=\mathbf{F}(\x)=(F_0(\x),F_1(\x))\trans$, where $\x=(x_0,x_1)\trans$.  Then
		\begin{eqnarray}\label{deffx}
		&&f(\x)=\sum_{j=0}^s {d\choose j} s_j x_0^{d-j}x_1^j=x_0^d \phi(\frac{x_1}{x_0}), \textrm{ where }\phi(z)=\sum_{j=0}^s {d\choose j} s_j z^j,\\
		\label{formF0}
		&&F_0(\x)=\sum_{j=0}^{d-1} {d-1\choose j} s_jx_0^{d-1-j}x_1^j =\frac{1}{d}\frac{\partial f}{\partial x_0},\\
		&&F_1(\x)=\sum_{j=0}^{d-1} {d-1\choose j}s_{j+1} x_0^{d-j-1}x_1^j=\frac{1}{d}\frac{\partial f}{\partial x_1}.\label{formF1}
		\end{eqnarray}
		\item
		\[x_0F_0(\x)+x_1F_1(\x)=f(\x)=\cS\times \otimes^d\x.\]
		\item
		The system (\ref{antifixpoint}) is
		\begin{equation}\label{antifixqub}
		F_0(\x)=\bar x_0, \quad F_1(\x)=\bar x_1.
		\end{equation}
		\item For $\mathbf{s}=(0,\ldots,0,s_d)\trans$ we have $\|\cS\|_{\sigma,\F}=|s_d|$.
	\end{enumerate}	
\end{lemma}
\begin{proof} \emph{1}.  The formula for $f(\x)=\sum_{j=0}^s {d\choose j} s_j x_0^{d-j}x_1^j$ follows from the observation that there exactly ${d\choose k}$ entries $\cS\in\rS^d\F^2$ that are equal to $s_k$.  Define $\phi(z)=\sum_{j=0}^s {d\choose j} s_j z^j\in\F[z]$.  Then $f(\x)=x_0^d\phi(\frac{x_1}{x_0})$.  This proves (\ref{deffx}).
	
	By definition $F_l(x_0,x_1)=\sum_{i_2,\ldots,i_d\in\{0,1\}} s_{l,i_2,\ldots,i_d}x_{i_2}\cdots x_{i_d}$ for $l=0,1$.
	Assume that exactly $j$ indices in the set $\{i_2,\ldots,i_d\}$ equal to $1$.  (Hence the other $d-1-j$ indices equal to $0$.)
	Then $x_{i_2}\cdots x_{i_d}=x_0^{d-1-j}x_1^j$.  The number of choices to have exactly $j$ indices in the set $\{i_2,\ldots,i_d\}$ equal to $1$ is $d-1\choose j$.
	Note that in this case $ s_{0,i_1,\ldots,i_d}=s_j$ and $ s_{1,i_1,\ldots,i_d}=s_{j+1}$.  These arguments prove the first part of (\ref{formF0}) and (\ref{formF1}) respectively.
	
	The equalities $F_l=\frac{1}{d}\frac{\partial f}{\partial x_l}$ for $l=0,1$ follow from the binomial identities
	\[{d\choose j}(d-j)=d{d-1 \choose j}, \quad {d\choose j}j=d{d-1\choose j-1}.\]
	\emph{2} is Euler's identity.
	
	\noindent
	\emph{3} is straightforward.
	
	\noindent
	\emph{4}. Assume that $\mathbf{s}=(0,\ldots,0,s_d)\trans$.  Then $\cS\times \otimes^d\x=s_d x_1^d$.  Hence $\|\cS\|_{\sigma,\F}=|s_d|$.\qed
\end{proof}	

The following theorem enables us to conclude that the geometric entanglement of qubits is polynomially computable. 
\begin{theorem}\label{computdqubspecnrm}  Let $\cS\in\rS^d\C^2$ and associate with $\cS$ the vector $\mathbf{s}=(s_0,\ldots,s_d)\trans\in \C^{d+1}$.
Assume that 
		\begin{equation}\label{Sassump}
		\mathbf{s}\ne (0,\ldots,0,s_d)\trans.
		\end{equation}
	Let $\phi(z)=\sum_{j=0}^s {d\choose j} s_j z^j$.  Then
	\begin{enumerate}
		\item 
		Define polynomials $p(z), q(z)$ and the rational rational function $r(z)$ as follows
		\begin{equation}\label{fzfor}
		p(z)=\sum_{j=0}^{d-1} {d-1\choose j}s_{j+1} z^j,\; q(z)=\sum_{j=0}^{d-1} {d-1\choose j} s_jz^j,\; r(z)=\frac{p(z)}{q(z)},\;
		\end{equation} 
		Then
		\begin{equation}\label{relpqrphi}
		p(z)=\frac{1}{d}\phi'(z),\; q(z)=\phi(z)-\frac{1}{d}z\phi'(z), \;r(z)=\frac{\phi'(z)}{d\phi(z)-z\phi'(z)}.
		\end{equation}
		The system (\ref{antifixqub}) for $(x_0,x_1)\trans\ne \0$ is equivalent to either
		\begin{equation}\label{antifixqub1}
		\bar z q(z)-p(z)=0,
		\end{equation}
		and $q(z)\ne 0$, i.e., $(x_0,x_1)\trans\ne (0,x_1)\trans$,  or  
		\begin{equation}\label{antifixqubinfty}
		s_{d-1}=0 \textrm{ and } s_d\ne 0 \;(z=\infty).
		\end{equation}
		\item Let
		\begin{equation}\label{defbarpqf}
		\bar p(z)=\sum_{j=0}^{d-1} {d-1\choose j}\bar s_{j+1} z^j, \bar q(z)=\sum_{j=0}^{d-1} {d-1\choose j} \bar s_jz^j,\bar r(z)=\frac{\bar p(z)}{\bar q(z)},g(z)=\bar r (r(z)).
		\end{equation}
		Then $g(z)=\frac{u(z)}{v(z)}$, where
		\begin{eqnarray}\label{defuz}
		u(z)= \sum_{j=0}^{d-1} {d-1\choose j}\bar s_{j+1} 
		\big(\sum_{k=0}^{d-1} {d-1\choose k} s_{k+1} z^k\big)^{j}\big(\sum_{k=0}^{d-1} {d-1\choose k}  s_kz^k\big)^{d-1-j},\\ 
		\label{defvz}
		v(z)= \sum_{j=0}^{d-1} {d-1\choose j}\bar s_{j} 
		\big(\sum_{k=0}^{d-1} {d-1\choose k} s_{k+1} z^k\big)^{j}\big(\sum_{k=0}^{d-1} {d-1\choose k} s_kz^k\big)^{d-1-j}.\quad
		\end{eqnarray} 
		Furthermore, the system (\ref{fixpointH}) for $(x_0,x_1)\trans\ne (0,x_1)\trans$ is equivalent to
		\begin{equation}\label{poleqfixpoint}
		zv(z)-u(z)=0,
		\end{equation}
		and $v(z)\ne 0$. 
		This is a polynomial equation of degree $(d-1)^2+1$ at most.   
		\item  Let 
		\begin{eqnarray}\label{defR}
		&&R=\{z\in\C, \;\bar z q(z)-p(z)=0\},\quad R'=R\cap\R,\\ 
		&&R_1=\{z\in\C,\;zv(z)-u(z)=0\},\quad R_1'=R_1\cap \R.\label{defR1}
		\end{eqnarray}
		Then $R'$ is the set of the real zeros of $zq(z)-p(z)=0$.  Furthermore, $R\subseteq R_1$.  In particular, if $z\in R$ and $q(z)=0$ then $z\in R_1$ and $v(z)=0$.
		\item The polynomial $zv(z)-u(z)$ is a zero polynomial if and only if one of the following conditions hold.  Either 
		\begin{equation}\label{geqiden}
		\mathbf{s}=A(\delta_{1(k+1)},\ldots,\delta_{(d+1)(k+1)}\trans) \textrm{ for } k\in[d-1], (f(\x)=A{d\choose k}x_0^{d-k}x_1^k),
		\end{equation}
		where $A$ is a nonzero scalar constant.  For this $\cS\in\rS^d\F^2$ we have
		\begin{equation}\label{geqidenSform}
		\|\cS\|_{\sigma,\F}=|A|{d \choose k} \big(1-\frac{k}{d}\big)^{\frac{d-k}{2}}\big(\frac{k}{d}\big)^{\frac{k}{2}}.
		\end{equation}
		Or $\cS$ has corresponding $\phi$ given by
		\begin{eqnarray}\label{geqiden1}
		&&\phi(z)=A (z+a)^p(z+b)^{d-p},\quad A\ne 0,\\ 
		&&a=e^{-s\i}c,\; b=-e^{-s\i}c^{-1},\;c\in\R\setminus\{0\}, \; s\in\R,\;p\in[d-1].\notag
		\end{eqnarray}
		Assume that $\cS\in\rS^d\R^2$ of the form \eqref{geqiden1}, i.e., $A\in\R$ and $s=0$.  Then $\|\cS\|_{\sigma,\R}$ is given by (\ref{realforspecnrm}), where  $R'$ is the set real solutions, (consisting of one or two real roots), of the quadratic equation $r(t)=t$.   
		
		Assume that $\cS\in\rS^d\C^2$ is of the form (\ref{geqiden1}).  Then $\|\cS\|_{\sigma}$ can be computed to an arbitrary precision as explained in \S\ref{sec:excepcase}.
		\item  Assume that $d>2$ and $\cS\in\rS^d\C^2$ is not of the form given in \emph{4}.    Then
		\begin{equation}\label{specnrmSform}
		\|\cS\|_{\sigma}= \max\left\{|s_d|, \max\{\frac{|q(z)|}{(1+|z|^2)^{\frac{d-2}{2}}}, z\in R\}\right\}.
		\end{equation}
		\item  Assume that $d>2$ and $\cS\in\rS^d\C^2$ is not of the form given in \emph{4}.  Then
		\begin{equation}\label{specnrmSform1}
		\|\cS\|_{\sigma}= \max\left\{|s_d|, \max\{\frac{|v(z)|^{\frac{1}{d}}}{(1+|z|^2)^{\frac{d-2}{2}}}, z\in R_1\}\right\}.
		\end{equation}
		\item  Assume that $\cS\in\rS^d\R^2$,  $d> 2$ and $\cS$ is not of the form given in \emph{4}. Then  
		\begin{equation}\label{realforspecnrm}
		\|\cS\|_{\sigma,\R}=\max\left\{|s_d|,  \max\{\frac{|q(t)|}{(1+|t|^2)^{\frac{d-2}{2}}}, t\in R'\}\right\}.
		\end{equation}
		\item  Assume that $\cS\in\rS^d\R^2$, $d> 2$ and $\cS$ is not of the form given in \emph{4}.  Then  
		\begin{equation}\label{realforspecnrm1}
		\|\cS\|_{\sigma,\R}= 
		\max\left\{|s_d|, \max\{\frac{|v(t)|^{\frac{1}{d}}}{(1+|t|^2)^{\frac{d-2}{2}}}, t\in R_1'\}\right\}.
		\end{equation}
		
	\end{enumerate}
\end{theorem}
\begin{proof}  We use the notations and the results of Lemma \ref{qubspecnrmlem}.

\noindent
\emph{1}.  The assumption  $\mathbf{s}\ne (0,\ldots,0,s_d)\trans$ is equivalent to the assumption that the polynomial $q(z)$ is not zero identically.  Observe next that if $x_0\ne 0$ then
	\[ z=\frac{x_1}{x_0},\; \bar z=\frac{\bar x_1}{\bar x_0},\; \frac{F_1(x_0,x_1)}{x_0^{d-1}}=p(z),\; \frac{F_0(x_0,x_1)}{x_0^{d-1}}=q(z).\]
	Recall that
	\begin{eqnarray*}
		dF_1=\frac{\partial f}{\partial x_1}=\frac{\partial (x_0^{d}\phi(\frac{x_1}{x_0}))}{\partial x_1}=x_0^{d-1}\phi'(\frac{x_1}{x_0}),\\
		dF_0=\frac{\partial f}{\partial x_0}=\frac{\partial (x_0^{d}\phi(\frac{x_1}{x_0}))}{\partial x_0}=dx_0^{d-1}\phi(\frac{x_1}{x_0})-x_1x_0^{d-2}\phi'(\frac{x_1}{x_0}).
	\end{eqnarray*}
	These equalities yield (\ref{relpqrphi}).
	
	Suppose first that (\ref{antifixqub}) holds.  Assume that $(x_0,x_1)\trans\ne \0$.
	Suppose first that $x_0\ne 0$.  Then $q(z)=\frac{F_0(x_0,x_1)}{x_0^{d-1}}=\frac{\bar x_0}{x_0^{d-1}}\ne 0$.  Furthermore
	\[\bar z= \frac{F_1(x_0,x_1)}{F_0(x_0,x_1)}=\frac{F_1(x_0,x_1)x_0^{-d+1}}{F_0(x_0,x_1)x_0^{-d+1}}=r(z).\]
	This shows the conditions (\ref{antifixqub1}).  Assume now that $x_0=0$. So $0=F_0(0,x_1)=s_{d-1}x_1^{d-1}$.  As $\x=(0,x_1)\ne \0$ it follows that $x_1\ne 0$.
	Hence $s_{d-1}=0$.  Furthermore $\bar x_1=F_1(0,x_1)=s_d x_1^{d-1}$.  Hence $s_d\ne 0$.  This shows (\ref{antifixqubinfty}).
	
	Conversely, suppose first that (\ref{antifixqubinfty}) holds.  Let $\x=(0,x_1)\trans$.  Then $F_0(\x)=0$ and $F_1(\x)=s_dx_1^{d-1}$.  Then choose $x_1$ to be a nonzero solution of $s_d x_1^{d-1}=\bar x_1$.  That is $|x_1|=|s_d|^{-\frac{1}{d-2}}$ and $\arg x_1=-\frac{\arg s_d}{d}$.  Thus $(0,x_1)\trans\ne \0$ is a solution of (\ref{antifixqub}).
	
	Assume second that (\ref{antifixqub1}) holds and $q(z)\ne 0$.  The above arguments show that there exists a nonzero solution to the equation $x_0^{d-1}q(z)=\bar x_0$.  Let $x_1=x_0 z, \x=(x_0,x_1)\trans$.  Then $F_0(x_0,x_1)=x_0^{d-1}q(z)=\bar x_0$.  Observe next \[F_1(x_0,x_1)=x_0^{d-1}p(z)=x_0^{d-1}q(z)\bar z=F_0(x_0,x_1)\bar z=\bar x_0 \bar z=\bar x_1.\] 
	Hence (\ref{antifixqub}) holds for 
	\begin{equation}\label{formbx}
	\x=x_0(1,z)\trans, \quad |x_0|=|q(z)|^{\frac{1}{2-d}},\;\arg x_0=-\frac{\arg q(z)}{d}.
	\end{equation}
	
	\noindent
	\emph{2}.  Assume that $\bar p(z), \bar q(z), \bar r(z) $ and $g(z)$  are defined in (\ref{defbarpqf}).   The definitions of $p(z),q(z)$ given by (\ref{fzfor}) yield the identities (\ref{defuz}) and (\ref{defvz}).
	Let
	\[u(z)=\sum_{k=0}^{(d-1)^2} u_kz^k, \quad v(z)=\sum_{k=0}^{(d-1)^2} v_kz^k.\]
	Then
	\begin{equation}\label{uvdfor}
	u_{(d-1)^2}=\sum_{j=0}^{d-1} {d-1\choose j}\bar s_{j+1}s_{d}^j s_{d-1}^{d-1-j},\;
	v_{(d-1)^2}= \sum_{j=0}^{d-1} {d-1\choose j}\bar s_{j}s_{d}^j s_{d-1}^{d-1-j}.
	\end{equation}
	Hence the polynomial $zv(z)-u(z)$ is of at most degree $(d-1)^2+1$.  Suppose $\x=(x_0,x_1)\trans, x_0\ne 0$ satisfies (\ref{fixpointH}): 
	\begin{equation}\label{fixpointH1}
	\overline{\mathbf{F}}(\bF(\x))=\x.
	\end{equation}
	Since we assumed that $x_0\ne 0$ it follows that 
	\[u(z)=\frac{\bar\bF_1(\bF(\x))}{x_0^{(d-1)^2}}, \quad v(z)=\frac{\bar\bF_0(\bF(\x))}{x_0^{(d-1)^2}}=x_0^{-(d-1)^2 +1}\ne 0.\]
	Hence the system (\ref{fixpointH1}) implies
	\[\bar r( r(z))=\frac{\bar\bF_1(\bF(\x))x_0^{-(d-1)^2}}{\bar\bF_0(\bF(\x))x_0^{-(d-1)^2}}=\frac{x_1}{x_0}=z.\]
	This yields (\ref{poleqfixpoint}), which is a polynomial equation of degree at most $(d-1)^2 +1$.
	
	\noindent
	\emph{3}.  As $\bar z= z$ for $z\in\R$ we deduce that $R'$ is the set of the real zeros of $zq(z)-p(z)=0$.  Assume that $z\in\R$.  Suppose first that $q(z)\ne 0$.  Part \emph{1} yields that there exists $(x_0,x_1)\trans, x_0\ne 0$ which satisfies  (\ref{antifixqub}).
	Hence $(x_0,x_1)\trans$ is a fixed point of $\bH$.  Thus $z=\frac{x_1}{x_0}\in R_1$ and $v(z)\ne 0$.  Suppose second that $q(z)=0$.  The equality $zq(z)-p(z)=0$ yields that $p(z)=0$. The equalities (\ref{defuz}) and (\ref{defvz}) show that $u(z)$ and $v(z)$ are homogeneous polynomials in variables $p(z),q(z)$.  Hence $u(z)=v(z)=0$.  In particular, $z\in R_1$. 
	
	\noindent
	\emph{4}.  The analysis of the exceptional cases is given in \S\ref{sec:excepcase}.
	
	\noindent
	\emph{5}.  Assume that $d>2$, and  $\mathbf{s}=(s_0,\ldots,s_d)\trans\in\C^{d+1}$ is the corresponding vector to $\cS=[\cS_{i_1,\ldots,i_d}]\in\rS^d\C^2$.
	Clearly, $\|\cS\|_{\sigma}\ge |\cS\times\otimes^d\e_1|=|s_{1,\ldots,1}|=|s_d|$.  
	
	Without loss of generality we may assume that $\cS\ne 0$.  Hence $\|\cS\|_{\sigma}>0$.
	Suppose first $\mathbf{s}=(0,\ldots,0,s_d)\trans$. Then $q(z)\equiv 0$ and $R=\emptyset$.   Hence (\ref{specnrmSform}) is equivalent to part \emph{4} of Lemma
	\ref{qubspecnrmlem}.
	
	Let $R_0=\{z\in R,\;q(z)=0\}$.  Clearly $|s_d|\ge \frac{|q(z)|}{(1+|z|^2)^{\frac{d-2}{2}}}$ for $z\in R_0$.  Thus it suffices to show that 
	\[\|\cS\|_{\sigma}= \max\left\{|s_d|, \max\{\frac{|q(z)|}{(1+|z|^2)^{\frac{d-2}{2}}}, z\in R\setminus R_0\}\right\}.\]
	
	The proof of Theorem \ref{estnumbanteig} yields that $(x_0,x_1)\trans$ is a critical point 
	of $f(\x):=\Re(\cS\times\otimes^d\x)$ on $\rS(2,\C)$ if and only if it satisfies  (\ref{antifixqub}).  Part \emph{1} yields that $(x_0,x_1)\trans\in \rS(2,\C)$ satisfies (\ref{antifixqub}) if and only if either $z=\frac{x_1}{x_0}\in R\setminus R_0$ or (\ref{antifixqubinfty}) holds.  Clearly $\|\cS\|_{\sigma}$ is a critical value of $f|\rS(2,\C)$.
	
	Suppose first that (\ref{antifixqubinfty}) holds.  Then the critical point of $f$ is $\y=(0,y_1)\trans, |y_1|=1$. The corresponding critical value of $f$ is $\pm |s_d|$, and $|s_d|\le \|\cS\|_{\sigma}$. 
	
	Assume second that $z\in R\setminus R_0$.
	Then $\x=(x_0,x_1)\trans$ is given by (\ref{formbx}) and satisfies (\ref{antifixqub}).  Clearly, $\|\x\|=|q(z)|^{\frac{1}{2-d}}\sqrt{1+|z|^2}$. (We do not assume that $\|\x\|=1$.)
	Let $\y=\frac{1}{\|\x\|}\x\in\rS(2,\C)$.  Then $\y$ is a critical point of $f|\rS(2,\C)$ which corresponds to the positive critical value 
	\begin{equation}\label{lamda:specnorm}
	\lambda_q(z)=\frac{1}{\|\x\|^{d-2}}=\frac{|q(z)|}{(\sqrt{1+|z|^2})^{d-2}}.
	\end{equation}
	Hence (\ref{specnrmSform}) holds.\\
	
	\noindent
	\emph{6}.  We now repeat the arguments of \emph{5} with corresponding modifications.  Let $R_{1,0}=\{z\in R_1, \;v(z)=0\}$.  Suppose that $z\in R_1\setminus R_{1,0}$.   Let $\x=(x_0,x_1)\trans$ and assume that $x_0\ne 0$ and $x_1=x_0 z$. 
	The assumption that $\bH(\x)=(H_0(\x),H_1(\x))=(x_0,x_1)$ yields that $H_0(\x)=x_0^{(d-1)^2}v(z)=x_0$.  Hence 
	\[|x_0|=|v(z)|^{-\frac{1}{(d-1)^2-1}}, \quad \arg x_0=-\frac{\arg v(z)}{(d-1)^2-1}.\]
	Clearly $H_1(\x)=x_0^{(d-1)^2}u(z)=x_0^{(d-1)^2}v(z)z=x_0z=x_1$.  Thus $\bH(\x)=\x$.  Let
	\begin{equation}\label{lamda:specnorm1}
	\lambda_v(z)=\frac{1}{\|\x\|^{d-2}}=(\frac{|v(z)|^{\frac{1}{(d-1)^2-1}}}{\sqrt{1+|z|^2}})^{d-2}=\frac{|v(z)|^{\frac{1}{d}}}{(\sqrt{1+|z|^2})^{d-2}}.
	\end{equation}
	The inequality (\ref{FHest2}) yields that $\lambda_v(z)\le \|\cS\|_{\sigma}$.  \emph{5}
	yields that if $|s_d|<\|\cS\|_{\sigma}$, then there exists $\x=(x_0,x_1)\trans, x_0\ne 0$ such that $\bF(\x)=\overline{\x}$ and $\|\cS\|_\sigma=\|\x\|^{-(d-2)}$.  Clearly, $\bH(\x)=\x$.  Hence (\ref{specnrmSform1}) holds.
	
	\noindent
	\emph{7}.  Assume that $d>2$, and  $\mathbf{s}=(s_0,\ldots,s_d)\trans\in\R^{d+1}$ is the corresponding vector to $\cS=[s_{i_1,\ldots,i_d}]\in\rS^d\R^2$.
	We now follow the notations and the arguments of \emph{5}.   It is enough to consider the case (\ref{Sassump}).  Then either $\|\cS\|_{\sigma}$ or $-\|\cS\|_{\sigma}$
	is a nonzero critical point of $\phi$ restricted to $\rS(2,\R)$.  Consider first a positive critical value of $\phi$.  The arguments of the proof of Theorem \ref{estnumbanteig} yields that such a critical point induces a nonzero solution to (\ref{antifixpoint}), which is equal to $\bF(\x)=\x$ since $\x=(x_0,x_1)\trans\in\R^2$.  For $x_0\ne 0$ we obtain that $t=z=\frac{x_1}{x_0}\in\R$.  In this case $t$ is a real solution of $p(t)-tq(t)=0, q(t)\ne 0$.  Consider second a negative critical values of $\phi$ restricted to $\rS(2,\R)$.  The arguments of the proof of Theorem \ref{estnumbanteig} yields that such a critical value induces a nonzero solution to $\bF(\x)=-\x$.   For $x_0\ne 0$ we obtain that $t=z=\frac{x_1}{x_0}\in\R$ satisfies
	the same equation $tq(t)-p(t)=0$.
	
	Vice versa, suppose that $t\in\R$ satisfies the equation $p(t)-tq(t)=0$ and $q(t)\ne 0$.  We claim that such $t$ induces a real solution to $\bF(\x)=\varepsilon\x$ for $\varepsilon\in\{-1,1\}$.
	As in \emph{1}  we need to find $x_0\in\R$ such that the equation $x_0^{d-1}q(t)=\varepsilon x_0$ has a nonzero real solution $x_0$.  Choosing $\varepsilon$ to be the sign of $q(t)$ we always have a positive solution $x_0$ to this equation.  Then we let $\x=x_0(1,t)\trans$.  Hence (\ref{realforspecnrm}) holds.
	
	\noindent
	\emph{8} follows from the arguments of \emph{6} and \emph{7}.\qed
\end{proof}

\begin{remark}\label{remark2}  It is straightforward to show using the arguments of the proof of Theorem \ref{computdqubspecnrm} that for $d>2$ every root of $zq(z)-p(z)=0, q(z)\ne 0$ corresponds to an eigenvector
	\begin{equation}\label{fixpointSFmap}
	\cS\times \otimes^{d-1}\y=\y, \quad \cS\in\rS^d\C^2, \y=(y_0,y_1)\trans \in\C^2, y_0\ne 0.
	\end{equation}
	Furthermore the inequality (\ref{FHest2}) holds.
\end{remark}
\section{The exceptional cases}\label{sec:excepcase}
In this section we discuss part \emph{4} of Theorem \ref{computdqubspecnrm}. 
Assume that $g(z)=\bar r(r(z))=z$ identically.  Recall that $r(z)$ can be viewed as a holomorphic map of the Riemann sphere.  The degree of this map is $\delta\in\N$ since $g$ is not a constant map.  Hence the degree of the map $g$ is $\delta^2$.  Since $g$ is the identity map and its degree is $1$.  Hence $\delta=1$ and $r(z)$ is a M\"obius map:
\[r(z)=\frac{az+b}{cz+d}, \quad ad-bc\ne 0.\]
Use the formula for $r(z)$ in (\ref{relpqrphi}) to deduce
\[\frac{1}{d}(\log \phi(z))'=\frac{1}{d}\frac{\phi'(z)}{\phi(z)}=\frac{az+b}{cz+d +z(az+b)}.\]
Let $l$ be the number of distinct roots of $\phi(z)$.  Then the logarithmic derivative of $\phi(z)$ has exactly $l$ distinct poles.  Compare that with the above formula of the logarithmic  derivative of $\phi$ we deduce that $\phi$ has either one (possibly) multiple root or two distinct roots.  

Assume first that $\phi(z)$ has one root of multiplicity $k$: $\phi(z)=A(z+a)^k$ for $k\in[d]$.  Then 
\[r(z)=\frac{k}{(d-k)z+d a}, \quad \bar r(z)=\frac{k}{(d-k)z+d \bar a}.\]
In this case $g(z)\equiv z$ if and only if $k\in[d-1]$ and $a=0$.
Clearly, if $\phi(z)=Az^k$, where $A\ne 0$, then $g(z)\equiv z$.
In this case $\cS\times\otimes^d \x=A{d\choose k} x_0^{d-k}x_1^k$.  To find $\|\cS\|_{\sigma,\F}$ we need to maximize $|A|{d\choose k}|x_0|^{d-k}|x_1|^k$ subject to
$|x_0|^2+|x_1|^2=1$.  The maximum is obtained for 
$|x_0|^2=1-\frac{k}{d}, |x_1|^2=\frac{k}{d}$.  This proves (\ref{geqidenSform}).

Assume now that $\phi(z)$ has two distinct zeros: $\phi(z)=A(z+a)^p(z+b)^q$, where $a\ne b$, $p,q\in\N$ and $p+q\le d$.  Then
\[r(z)=\frac{(z+a)^{p-1}(z+b)^{q-1}\big((p+q)z+pb+qa\big)}{(z+a)^{p-1}(z+b)^{q-1}\big(d(z+a)(z+b)-z\big((p+q)z+pb+qa\big)\big)}.\]
Suppose first that $p+q<d$.  In order that $r(z)$ will be a M\"obius transformation
we need to assume that $(p+q)z+pb+qa$ divides $(z+a)(z+b)$.  This is impossible,
since $\phi'$ has exactly $p+q-2$ common roots with $\phi$.  Hence we are left with the case $p+q=d$.  In this case
\begin{equation}\label{rxfroexcase}
r(z)=\frac{dz+\alpha}{\beta z+dab}, \quad \alpha=pb+qa, \beta=d(a+b)-\alpha, p+q=d.
\end{equation}
The assumption that $g(z)\equiv z$ is equivalent to the following matrix equality
\begin{equation}\label{defmatA}
\bar A A=\gamma^2 I_2, \quad A=\left[\begin{array}{cc}d&\alpha\\\beta&dab\end{array}\right] \quad \gamma\ne 0.
\end{equation}
Taking the determinant of the above identity we deduce that $\gamma^4=|\det A|^2>0$.  So $\gamma^2=\pm \tau^{-2}$ for some $\tau>0$.
Let $B=\tau A$.  Suppose first that $\gamma^2=\tau^{-2}$.  Then $\bar B$ is the inverse of $B$.  So $\det B=\delta, |\delta|=1$.  Then
\begin{equation}\label{analeq}
dab=\delta d,\; d=\delta d\bar a\bar b,\;-\alpha=\delta\bar \alpha, -\beta=\delta \bar\beta.
\end{equation}
Hence $ab=\delta$.  

We next observe that if we replace $\phi(z)$ with $\phi_s(z):=\phi(e^{s\i }z)$ for any $s\in\R$ we will obtain the following relations
\[p_s(z)=e^{s \i}p(e^{s \i}z), q_s(z)=q(e^{s \i}z), \overline{p_s}(z)=e^{-s \i}p(e^{-s \i}z), \overline{q_s}(z)=\bar q(e^{-s \i}z).\]
Let 
\[r_s(z)=\frac{p_s(z)}{q_s(z)}, \;\overline{r_s}(z)=\frac{\overline{p_s}(z)}{\overline{q_s}(z)},\; g_s(z)=\overline{r_s}(r_s(z)).\]
A straightforward calculation shows that $g_s(z)=z$ for all $s\in\R$.  Note the two roots of $\phi_s(z)$ are $-ae^{-s\i}, -be^{-s\i}$.  Hence we can choose $s$ such that $\delta=-1$.  Assume for simplicity of the exposition this condition holds for $s=0$, i.e., for $\phi$.  So $\alpha$ and $\beta$ are real.  In particular $a+b$ is real.  So $b=-a^{-1}$ and $a-a^{-1}$ is real.  Hence $a$ is real and also $b$ is real.

Suppose now that $\gamma^2=-\tau^2$.  Then $\bar B=-B^{-1}$. So $\det B=\delta, |\delta|=1$.  Then
\begin{equation}\label{analeq1}
dab=-\delta d,\; d=-\delta d\bar a\bar b,\;\alpha=\delta\bar \alpha, \beta=\delta \bar\beta.
\end{equation}
By considering $\phi_t$ as above we may assume that $\delta=1$ which gives again that $a\in\R\setminus\{0\}$ and $b=-a^{-1}$.  This proves (\ref{geqiden1}).

Vice versa, assume that $\phi(z)$ is of the form (\ref{geqiden1}), where $a\in\R\setminus\{0\}$ and $b=-\frac{1}{a}$.  We claim that $\bar r(r(z))\equiv z$.
The above arguments show that (\ref{rxfroexcase}) holds.  Furthermore
\begin{equation}\label{excforalphbet}
ab=-1,\;\alpha=(d-p)a -\frac{p}{a}, \beta=pa -\frac{d-p}{a}, \; p\in[d-1].
\end{equation}
Consider the matrix $A$ given by (\ref{defmatA}). Note the trace of this matrix is zero.  We claim that $A$ is not singular.  Indeed
\begin{equation}\label{formdetA}
\det A=-(d^2 +\alpha\beta)=-(d^2 - (d-p)^2 -p^2 +p(d-p)(a^2+a^{-2}))=-p(d-p)(a+a^{-1})^2.
\end{equation}
Hence $A$ has two distinct eigenvalues $\{\gamma,-\gamma\}$ and is diagonalizable.
As $A$ is a real matrix we get that $\bar A A= A^2=\gamma^2 I_2$.  Therefore 
$\bar r(r(z))\equiv z$.  As $\bar r_s(r_s(z))\equiv z$ for each $s\in\R$ we deduce that for each $\phi$ of the form (\ref{geqiden1}) $zq(z)-p(z)$ is a zero polynomial.

We now give an algorithm to compute $\|\cS\|_{\sigma,\F}$ corresponding to $\phi$ of the form (\ref{geqiden1}).  Clearly
\begin{equation}\label{formFxcom}
\cS\times\otimes^d\x=A(x_1+ce^{-s\i}x_0)^p(x_1-c^{-1}e^{-s\i}x_0)^{d-p}
, \; c\in\R\setminus\{0\},p\in[d-1],\x=(x_0,x_1)\trans.
\end{equation}
Assume first that $\cS\in\rS^d\R^2$.  It is straightforward to show that $A\in\R$ and $e^{s\i}=\pm 1$ unless $c\ne \pm 1$ and $d=2p$.  These two cases are equivalent to the assumption that 
\begin{equation}\label{specformFx1}
\cS\times\otimes^d\x= A(x_1+cx_0)^p(x_1-c^{-1}x_0)^{d-p}, \; c\in\R\setminus\{0\},d\in[d-1],\x=(x_0,x_1)\trans.
\end{equation}
In the case that $c=\pm 1$ and $d=2p$ we have another possibility
\begin{equation}\label{specformFx2}
\cS\times\otimes^d\x=A(x_1^2+x_0^2)^p, \quad p\in\N.
\end{equation}

We first consider the case (\ref{specformFx2}).  A straightforward calculation shows that $r(z)\equiv z$.  Clearly $\bar r(z)=r(z)$.  Thus the polynomials $zq(z)-z$ and $zv(z)-z$ are zero polynomials.  In that case $\|\cS\|_{\sigma,\R}=|A|$ and $\cS\times \otimes^{2p}\x$ has value $A$ for all $\x\in\rS(2,\R)$. 

Next we consider the case (\ref{specformFx1}).  Then $r(z)$ is given by 
\begin{equation}\label{formrzcase1}
r(z)=\frac{dz+\alpha}{\beta z -d}, \quad \alpha=(d-p)c-pc^{-1}, \beta=pc -(d-p)c^{-1}.
\end{equation}
Hence the equation $r(z)=z$ boils down to the quadratic equation
\begin{equation}\label{realquadeq}
\beta z^2-2d z-\alpha=0.
\end{equation}
The equality (\ref{formdetA}) yields that $d^2+\alpha\beta>0$.  Hence the above quadratic equation has two or one real roots.  Let $R'$ the set of the roots of (\ref{realquadeq}).  Then part \emph{6} of Theorem \ref{computdqubspecnrm} implies the equality (\ref{realforspecnrm1}).

Assume now that $\rS_s\in\rS^d\C^2$ is induced by $\phi$ of the form (\ref{geqiden1}). The equality (\ref{formFxcom}) yields that $\|{\cS}_s\|_{\sigma}=\|\cS_0\|_{\sigma}$.  Thus it is enough to find $\|\cS_0\|_{\sigma}$.  For simplicity of notation we let $\cS=\cS_0$.  In this case $\r(z)$ is of the form (\ref{formrzcase1}).  The equation $r(z)=\bar \z=x-y\i$ boils down to one equation in two variables $x,y\in\R$:
\begin{equation}\label{xyeq}
\beta(x^2+y^2)-2dx-\alpha=0, \quad z=x+y\i.
\end{equation}
Suppose first that $\beta=0$.   Then  $x=-\frac{\alpha}{2d}$ and $y$ is a free variable.  For $\beta\ne 0$ the pair $(x,y)$ lie on a circle
\[ (x-\frac{d}{\beta})^2 +y^2=\frac{d^2+\alpha\beta}{\beta^2}.\] 
This means that we replaced the problem of finding the maximum of $|\cS\times \otimes^d\x|$ over the three dimensional sphere $\rS(2,\C)$ with another problem with three (or two) real parameters $|x_0|,x,y$ satisfying the condition
$|x_0|^2(1+x^2+y^2)=1$. 

We now suggest following simple approximations to find $\|\cS\|_{\sigma}$ using the case \emph{4} (or \emph{5}) of Theorem \ref{computdqubspecnrm}.   Assume first that $A=1$, i.e., $\cS$ corresponds to $\phi(z)=(z+c)^p(z-c^{-1})^{d-p}$.  Choose $\varepsilon>0$ such that our desired approximation should be not more than $2\epsilon$.  Let $\phi(z,\varepsilon)=\phi(z)+\varepsilon$.
Let $\cS(\varepsilon)\in\rS^d\R^2$ be the symmetric tensor corresponding to $\phi(z,\varepsilon)$.  It is straightforward to show that 
\[\|\cS(\varepsilon)-\cS\|_{\sigma}=\|\cS(\varepsilon)-\cS\| = \varepsilon.\] 
Hence  $\|\cS\|_{\sigma}\in [\|\cS(\varepsilon)\|_{\sigma}-\varepsilon,\|\cS(\varepsilon)\|_{\sigma}+\varepsilon]$.  

Observe next that $\phi'(z,\varepsilon)=\phi'(z)$.  Hence a common root of $\phi(z,\varepsilon)$ and 
$\phi'(z,\varepsilon)$ can not be a common root of $\phi(z)$ and $\phi'(z)$.  Hence
$\phi(z,\varepsilon)$ and $\phi'(z,\varepsilon)$ can have at most one common root.  As $d\ge 3$ we deduce that
$\cS(\varepsilon)$ does not satisfy the assumptions of part \emph{3} of Theorem \ref{computdqubspecnrm}.  Apply the case \emph{4} (or \emph{5}) of Theorem \ref{computdqubspecnrm} to compute $\|\cS(\varepsilon)\|_{\sigma}$ within precision $\varepsilon$.  The value of  $\|\cS(\varepsilon)\|_{\sigma}$ gives the value of $\|\cS\|_{\sigma}$ with precision $2\varepsilon$.

A disadvantage of this approximation that $\|\cS(\varepsilon)\|$ depends on $\varepsilon$.  We now consider the case where $S\in\rS^d\R^2$ corresponds to  $\phi(z)=\frac{1}{A}(z+c)^p(z-c^{-1})^{d-p}$, where $A>0$ and $\|\cS\|=1$.
Let $\cT=[\cT_{i_1,\ldots,i_d}]\in\rS^d\R^2$ corresponding to $\phi(z)=1$.  That is $\cT_{1,\ldots,1}=1$ and all other entries of $\cT$ are zero.    Then we let 
\[\cS(\varepsilon)=\frac{1}{\|\cS+\varepsilon\cT\|}(\cS+\varepsilon \cT).\]
It is not difficult to show that $\|\cS(\varepsilon)-\cS\|\le \varepsilon$.  Hence $\|\cS(\varepsilon)-\cS\|_{\sigma}\le \varepsilon$ and $\|\cS\|_{\sigma}\in [\|\cS(\varepsilon)\|_{\sigma}-\varepsilon, \|\cS(\varepsilon)\|_{\sigma}+\varepsilon]$.

In the following three examples below we consider $\cS(m)\in \rS^d\R^2$ corresponding to 
\begin{equation}\label{examexceptS}
\phi(z,m)=\frac{1}{A(m)}(z+1)^m(z-1)^m=\frac{1}{A(m)}(z^2-1)^m, \; \|\cS(m)\|=1,\; d=2m,
\end{equation}
where $A(m)>0$.
Note that $\cS(m)\times \otimes^{2m}\x=\frac{1}{A(m)}(x_1^2-x_0^2)^m$, and $|x_1^2-x_0^2|\le |x_1|^2 +|x_0|^2$, and equality holds if $x_1$ and $\i x_0$ have the same
argument. Hence $\|\cS\|_{\sigma}=\frac{1}{A(m)}$.  Note that $|x_1^2-x_0^2|=x_1^2+x_0^2$ for $x=(x_0,x_1)\trans \in\R^2$ and $x_0x_1=0$.  Hence $\|\cS(m)\|_{\sigma,\R}=\|\cS(m)\|_{\sigma}=\frac{1}{A(m)}$.

Let $\cS^{\varepsilon}(m)$ be the perturbation of $\cS(m)$ given by 
\[\frac{1}{A(m)}\big((z^2-1)^m -(z^{2m}+(-1)^m)\big)+\frac{1}{A(m)}\big(\sqrt{1-\varepsilon}z^{2m} +(-1)^m\sqrt{1+\varepsilon}\big).\]  
Note that $\|\cS^{\varepsilon}(m)\|=1$.  Furthermore
\[\|\cS(m)-\cS^{\varepsilon}(m)\|=\frac{1}{A(m)}\sqrt{\big(\sqrt{1-\varepsilon}-1\big)^2 +\big(\sqrt{1+\varepsilon}-1\big)^2}\approx \frac{\varepsilon}{\sqrt{2}A(m)}.\] 
\begin{exm}\label{special:case:exm:d:4} 
	\[d=2m=4, \quad A(2)=  \sqrt{\frac{8}{3}}, \quad \|\cS(2)\|_{\sigma}=\sqrt{\frac{3}{8}}\approx 0.61237243569.\]
	For $\varepsilon=\{0.5,0.1,0.05,0.01,0.005,0.001,0.0005,0.0001\}$, we calculate the real and complex spectral norms for $\cS^{\varepsilon}(2)$, the computational results are shown in Table \ref{table:special:case:exm:d:4}. 	
	
	\begin{table}\caption{Computational results for Example \ref{special:case:exm:d:4}. } \label{table:special:case:exm:d:4}
		\centering
		\begin{scriptsize}
			\begin{tabular}{|c|c|c||c|c|c|c|c|c|c|c|} \hline
				$\varepsilon$ &   $\|\cS^{\varepsilon}(2)\|_{\sigma,\R}$ & $\|\cS^{\varepsilon}(2)\|_{\sigma}$  & 	$\varepsilon$ &   $\|\cS^{\varepsilon}(2)\|_{\sigma,\R}$ & $\|\cS^{\varepsilon}(2)\|_{\sigma}$  \\ \hline 
				0.5& 0.75 & 0.75 & 0.1 & 0.64226 & 0.64226 \\ \hline 
				0.05& 0.62750 & 0.62750 & 0.01 & 0.61543 & 0.61543 \\ \hline 
				0.005& 0.61390 & 0.61390 & 0.001 & 0.61268 & 0.61268 \\ \hline 	
				0.0005& 0.61253 & 0.61253 & 0.0001 & 0.61240 & 0.61240\\ \hline 				
			\end{tabular} 
		\end{scriptsize}
	\end{table}	 
	
\end{exm}

\begin{exm}\label{special:case:exm:d:6}
	\[d=2m=6, \quad A(3)=  \frac{4}{\sqrt{5}}, \quad \|\cS(3)\|_{\sigma}=\frac{\sqrt{5}}{4}\approx 0.55901699437.\]
	For $\varepsilon=\{0.5,0.1,0.05,0.01,0.005,0.001,0.0005,0.0001\}$, we calculate the real and complex spectral norms for $\cS^{\varepsilon}(3)$, the computational results are shown in Table \ref{table:special:case:exm:d:6}.

	\begin{table}\caption{Computational results for Example \ref{special:case:exm:d:6}. } \label{table:special:case:exm:d:6}
		\centering
		\begin{scriptsize}
			\begin{tabular}{|c|c|c||c|c|c|c|c|c|c|c|} \hline
				$\varepsilon$ &   $\|\cS^{\varepsilon}(3)\|_{\sigma,\R}$ & $\|\cS^{\varepsilon}(3)\|_{\sigma}$  & 	$\varepsilon$ &   $\|\cS^{\varepsilon}(3)\|_{\sigma,\R}$ & $\|\cS^{\varepsilon}(3)\|_{\sigma}$  \\ \hline 
				0.5& 0.68465  &  0.68465 & 0.1 &  0.58630 & 0.58630  \\ \hline 
				0.05&  0.57282 &  0.57282 & 0.01 & 0.56181  & 0.56181 \\ \hline 
				0.005& 0.56041  & 0.56041  & 0.001 & 0.55930  & 0.55930  \\ \hline 	
				0.0005& 0.55916 & 0.55916  & 0.0001 & 0.55904  & 0.55904 \\ \hline 				
			\end{tabular} 
		\end{scriptsize}
	\end{table}	 
	
\end{exm}

\begin{exm}\label{special:case:exm:d:8}
	\[d=2m=8, \quad A(4)=  \sqrt{\frac{128}{35}}, \quad \|\cS(4)\|_{\sigma}=\sqrt{\frac{35}{128}}\approx 0.52291251658.\]
	For $\varepsilon=\{0.5,0.1,0.05,0.01,0.005,0.001,0.0005,0.0001\}$, we calculate the real and complex spectral norms for $\cS^{\varepsilon}(4)$, the computational results are shown in Table \ref{table:special:case:exm:d:8}. 	
	
	\begin{table}\caption{Computational results for Example \ref{special:case:exm:d:8}. } \label{table:special:case:exm:d:8}
		\centering
		\begin{scriptsize}
			\begin{tabular}{|c|c|c||c|c|c|c|c|c|c|c|} \hline
				$\varepsilon$ &   $\|\cS^{\varepsilon}(4)\|_{\sigma,\R}$ & $\|\cS^{\varepsilon}(4)\|_{\sigma}$  & 	$\varepsilon$ &   $\|\cS^{\varepsilon}(4)\|_{\sigma,\R}$ & $\|\cS^{\varepsilon}(4)\|_{\sigma}$  \\ \hline 
				0.5&  0.64043  & 0.64043   & 0.1 &  0.54844  & 0.54844   \\ \hline 
				0.05&  0.53583  &   0.53583 & 0.01 &  0.52552  & 0.52552  \\ \hline 
				0.005&  0.52422  &  0.52422  & 0.001 &   0.52317  &  0.52317   \\ \hline 	
				0.0005&  0.52304 &  0.52304  & 0.0001 &  0.52294  & 0.52294   \\ \hline 				
			\end{tabular} 
		\end{scriptsize}
	\end{table}	 
	
\end{exm}

\section{The case $n>2$}\label{sec:nge3}
We now discuss briefly the complexity of finding $\|\cS\|_{\sigma}$, where $\cS\in \rS^d\C^n$, for a fixed $n>2$ and $d\in\N$.  Recall that $\rS^d\C^n$ is a vector space of dimension ${n+d-1\choose n-1}=O(d^{n-1})$.  As discussed in \S\ref{sec:critpts} we need to find the nonzero solutions of (\ref{antifixpoint}).
To use algebraic geometry it is more convenient to consider the system (\ref{fixpointH}).  
The proof of Theorem \ref{estnumbanteig}  that the number of different solutions of the system (\ref{fixpointH}) counted with multiplicities  is  $(d-1)^{2n}$ if $\cS$ is not singular.  Let $\mathcal{A}_d=[\mathcal{A}_{i_1,\ldots,i_d}]$ be the \emph{identity} symmetric tensor, i.e.,
$\mathcal{A}_{i_1,\ldots,i_d}=1$ if $i_1=\cdots=i_d$ and $\mathcal{A}_{i_1,\ldots,i_d}=0$ otherwise.
Clearly $\mathcal{A}_d$ is nonsingular.
Then $\bH(\y)=\mathcal{A}_{(d-1)^2+1}\times \otimes^{(d-1)^2}\y$.  Hence the system $\bH(\y)=\y$ has exactly $(d-1)^{2n}$ distinct solutions.  Assume that $\cS(t)\in \rS^d\C^n$ is a continuous function in $t\in[0,1]$ such that $\cS(0)=\mathcal{A}_d$ and $\cS(1)=\cS$.  Assume that $\cS$ is nonsingular.  Since the variety of the singular symmetric tensors in $\rS^d\C^n$ is of codimension at least $1$ \cite{FO14}, with probability $1$
each $\cS(t), t\in (0,1)$ is nonsingular.  Thus we can apply the homotopy method of finding all  $(d-1)^{2n}$ fixed points of  $\bH(\y,t):=\overline{\cS(t)}\times \otimes^{d-1}(\cS(t)\times\otimes^{d-1}\y)$ \cite{BHSW06,BCSS98}.  On the other hand, if $\cS\in\rS^d\C^n$ is singular then while carrying out the homotopy method we either have that for some $\cS(t)$ the set of fixed points of $\bH(\y,t)$ is infinite, or some of the fixed points ``escape to infinity". 
(This follows from the fact that homogenizing the system $\bH(\y)-\y=\0$ to $\bH(\y)-y_0^{(d-1)^2}\y=\0$, where $(y_0,\y\trans)\trans\in \P\C^{n+1}$, we get a solution $(0,\y\trans)\trans\in \P\C^{n+1}$.
This corresponds to a solution of $\bH(\y)-\y=\0$ at infinity.)
That is, the homotopy method fails.

In particular we conjecture:
\begin{con}\label{polconjspecnrmS}  Assume that $2<n\in\N$ is fixed.  Suppose that $\cS\in\rS^d\C^n$ is not singular.  Then with probability $1$ we can find in polynomial time in $d$
	all fixed points of $\bH$ within $\varepsilon$ approximation.  In particular, we can compute $\|\cS\|_{\sigma}$ within $\varepsilon$ approximation  in polynomial time in $d$  with probability $1$.
\end{con}

We now discuss briefly the computational aspect of the problem if a given $\cS\in\rS^d\C^n$ is nonsingular.  This problem is equivalent to the problem if the homogeneous system $\bF(\y)=\0$ has only the trivial solution.   As we discussed above, this is equivalent to the statement that the system $\bF(\y)-\y=\0$ has $(d-1)^n$ distinct isolated solutions counting with
multiplicities.  Let $\cS(t)\in\rS^d\C^n, t\in [0,1]$ be defined as above.  Then, as explained above, $\cS(t)$ is nonsingular for $t\in [0,1)$ with probability $1$. 
Thus we can apply the homotopy method of finding all  $(d-1)^{n}$ fixed points of  $\bF(\y,t):=\cS(t)\times\otimes^{d-1}\y$ \cite{BHSW06, BCSS98}. 
If we can carry out the homotopy method for $t\in[0,1]$ then $\cS$ is nonsingular.  On the other hand, if $\cS\in\rS^d\C^n$ is singular then while carrying out the homotopy method we either have that for some $\cS(t)$ the set of fixed points of $\bF(\y,t)$ is infinite, or some of the fixed points ``escape to infinity". 
Thus we conjecture that with probability $1$ we can decide if $\cS$ is nonsingular in polynomial time.

\appendix
\section{Examples}\label{sec:examples}
In this section, we give some numerical examples of applications of Theorem \ref{computdqubspecnrm}. All the examples are real symmetric $d$-qubits since we expect to compute both the complex and real spectral norms of a given tensor $\cS$. 
In Theorem \ref{computdqubspecnrm}, we showed that the spectral norm of a given $d$-qubit can be found by solving the polynomial equation $zv(z)-u(z) =0$ (see (\ref{poleqfixpoint})), which is of degree $(d-1)^2+1$ at most.  In what follows that we assume that the polynomial $zv(z)-u(z)$ is not a zero polynomial.  (That is, we are not dealing with the exceptional cases that are discussed in \S\ref{sec:excepcase}.)  

In this paper we use Bertini \cite{BHSW06} (version 1.5, released in 2015), which is a well developed software to find the set $R_1$ and its real subset $R_1'$, which are given by (\ref{defR1}). We also find the subsets $R$  and $R'$,  which are given by (\ref{defR}). We calculate the parameters $\lambda_q$ and $\lambda_v$ which are given in (\ref{lamda:specnorm}) and (\ref{lamda:specnorm1}) respectively. We use formulas (\ref{specnrmSform}), (\ref{specnrmSform1}) to find $\|\cS\|_{\sigma}$, and formulas (\ref{realforspecnrm}), (\ref{realforspecnrm1}) to find $\|\cS\|_{\sigma,\R}$.


All the computation are implemented with Matlab R2012a on a MacBook Pro 64-bit OS X (10.9.5) system with 16GB memory and 2.3 GHz Intel Core i7 CPU. In the display of the computational results, only four decimal digits are shown. The default parameters in Bertini are used to solve the polynomial equation $zv(z)-u(z) =0$. 
Since all examples only take few seconds we will not show the computing time. 

In the following examples we specify only the nonzero entries of the symmetric tensor $\cS$.  We also assume that the symmetric tensor $\cS\in\rS^d\R^2$ is a state, i.e., $\|\cS\|=1$.

\begin{exm}\emph{\cite{Nie14}} \label{exm:1:d:qubit} Given tensor $\cS\in\rS^3\R^2$ with: 
	$$\cS_{1,1,1} =  0.3104, ~\cS_{2,2,2} =   0.2235,~\cS_{1,1,2}=-0.4866,~\cS_{1,2,2}=-0.2186.$$ 
	The polynomial $zv(z)-u(z)$ has degree 5. It has 5 roots, 3 of them are real and the other 2 are complex. The computational results are shown in Table \ref{table:exm:1:d:qubit}.  Clearly $|s_{3}| = |\cS_{2,2,2}|=0.2235$, and we have $R=R_1$, $R'=R_1'$. By using formulas (\ref{specnrmSform}), (\ref{specnrmSform1}), (\ref{realforspecnrm}), (\ref{realforspecnrm1}), we get:
	\begin{equation*}
	\begin{aligned}
	\|\cS\|_{\sigma~~} &  = \max\{|s_3|, \lambda_q(z), z\in R \} = \max\{|s_3|, \lambda_v(z), z\in R_1 \} = 0.7027,\\
	\|\cS\|_{\sigma,\R} & = \max\{|s_3|,\lambda_q(z),z\in R'\}= \max\{|s_3|,\lambda_v(z),z\in R_1'\} = 0.6205.
	\end{aligned}
	\end{equation*}

	\begin{table}\caption{Computational results for Example \ref{exm:1:d:qubit}. } \label{table:exm:1:d:qubit}
		\centering
		\begin{scriptsize}
			\begin{tabular}{|c|r|c|c|c||c|r|c|c|c|c|} \hline
				No. &  $z$\quad\quad\quad\quad  & $\F$ & $\lambda_q(z)$ & $\lambda_v(z)$   & No. &  $z$\quad\quad\quad\quad  & $\F$ &   $\lambda_q(z)$ & $\lambda_v(z)$ \\ \hline 
				1& 0.9157 & $\R$   &  0.5635  & 0.5635 & 4 &  
				-0.0240+ 0.7928i &$\C$&      0.7027 & 0.7027\\ \hline              
				2& -5.9838  & $\R$ &      0.2791  & 0.2791& 5& -0.0240 - 0.7928i &$\C$&       0.7027 & 0.7027 \\ \hline             
				3& -0.4063  & $\R$ &     0.6205 &  0.6205  & &&&& \\ \hline             
			\end{tabular} 
		\end{scriptsize}
	\end{table}	
	
\end{exm}

\begin{exm}  \emph{\cite{FLTN16}}  \label{exm:1:B:d:qubit} Given tensor $\cS\in\rS^3\R^2$ with: 
	$$\cS_{1,1,2} =  \frac{1}{2}, ~\cS_{2,2,2} = -\frac{1}{2}.$$
	The polynomial $zv(z)-u(z)$ has degree 4.  It has 4 roots, 2 of them are real and the other 2 are complex.  The computational results are shown in Table \ref{table:exm:1:B:d:qubit}. Clearly $|s_{3}| = |\cS_{2,2,2}|=0.5$, and we have $R=R_1$, $R'=R_1'$. By using formulas (\ref{specnrmSform}), (\ref{specnrmSform1}), (\ref{realforspecnrm}), (\ref{realforspecnrm1}), we get:	
	\begin{equation*}
	\begin{aligned}
	\|\cS\|_{\sigma~~} & = \max\{|s_3|, \lambda_q(z), z\in R \} = \max\{|s_3|, \lambda_v(z), z\in R_1 \} =  0.7071,\\
	\|\cS\|_{\sigma,\R} & = \max\{|s_3|,\lambda_q(z),z\in R'\} = \max\{|s_3|,\lambda_v(z),z\in R_1'\} =0.5.
	\end{aligned}
	\end{equation*}

	\begin{table}\caption{Computational results for Example \ref{exm:1:B:d:qubit}. } \label{table:exm:1:B:d:qubit}
		\centering
		\begin{scriptsize}
			\begin{tabular}{|c|c|c|c|c||c|c|c|c|c|} \hline
				No. &  $z$   & $\F$ &  $\lambda_q(z)$ & $\lambda_v(z)$    & No. &  $z$   & $\F$ &  $\lambda_q(z)$ & $\lambda_v(z)$   \\ 
				\hline 1 &  0.5774 & $\R$ & 0.5   & 0.5   & 3 &    -0.5774 & $\R$ & 0.5   & 0.5      
				\\  \hline
				2 &    1i & $\C$ & 0.7071 & 0.7071  & 4 &  -1i & $\C$ &  0.7071 & 0.7071  \\ \hline 
			\end{tabular} 
		\end{scriptsize}
	\end{table}	
\end{exm}

\begin{exm}\label{exm:2:d:qubit} Given tensor $\cS\in\rS^3\R^2$ with: 
	$$\cS_{1,1,1} = \frac{1}{\sqrt{5}}, ~\cS_{2,2,2} =  \frac{1}{2\sqrt{5}},~\cS_{1,1,2}=-\frac{1}{2\sqrt{5}},~\cS_{1,2,2}=-\frac{1}{\sqrt{5}}.$$ 
	The polynomial $zv(z)-u(z)$ has degree  5.  It has 5 roots, 3 of them are real and the other 2 are complex.  The computational results are shown in Table \ref{table:exm:2:d:qubit}.  Clearly $|s_{3}| = |\cS_{2,2,2}|=\frac{1}{2\sqrt{5}}$, and we have $R=R_1$, $R'=R_1'$. By using formulas (\ref{specnrmSform}), (\ref{specnrmSform1}), (\ref{realforspecnrm}), (\ref{realforspecnrm1}), we get:
	\begin{equation*}
	\begin{aligned}
	\|\cS\|_{\sigma~~} & = \max\{|s_3|, \lambda_q(z), z\in R \} = \max\{|s_3|, \lambda_v(z), z\in R_1 \} =  0.7071,\\
	\|\cS\|_{\sigma,\R} & = \max\{|s_3|,\lambda_q(z),z\in R'\} = \max\{|s_3|,\lambda_v(z),z\in R_1'\} =0.5000.
	\end{aligned}
	\end{equation*}

	\begin{table}\caption{Computational results for Example \ref{exm:2:d:qubit}. } \label{table:exm:2:d:qubit}
		\centering
		\begin{scriptsize}
			\begin{tabular}{|c|r|c|c|c||c|r|c|c|c|} \hline
				No. &  $z$\quad\quad   & $\F$ & $\lambda_q(z)$ & $\lambda_v(z)$ & No. &  $z$\quad & $\F$ & $\lambda_q(z)$ & $\lambda_v(z)$ \\ 
				\hline
				1 & 1.2413   & $\R$ & 0.5000 & 0.5000 & 4 &  -1i  & $\C$ & 0.7071 & 0.7071  \\  \hline
				2 &   -0.1558   & $\R$  & 0.5000 & 0.5000 & 5 &  1i  & $\C$  & 0.7071 & 0.7071  \\ \hline
				3 & -2.5855  &  $\R$ & 0.5000 & 0.5000 &   &     & &   &\\  \hline 
			\end{tabular} 
		\end{scriptsize}
	\end{table} 
	
\end{exm}

\begin{exm} \label{exm:3:d:qubit} Given tensor $\cS\in\rS^4\R^2$ with: 
	$$\cS_{1,1,1,1} = \frac{2}{11};\cS_{1,1,1,2} = -\frac{1}{11};\cS_{2,2,1,1}= \frac{4}{11}; \cS_{2,2,2,1} = -\frac{2}{11}; \cS_{2,2,2,2} = \frac{1}{11}.$$ 
	The polynomial $zv(z)-u(z)$ has degree 10.  It has 10 roots, 4 of them are real and the other 6 are complex.  The computational results are shown in Table \ref{table:exm:3:d:qubit}.  Clearly $|s_{4}| = |\cS_{2,2,2,2}|=\frac{1}{11}$. For the 10 roots $z$, (\ref{antifixqub1}) fails to satisfy for $z_i,~i\in\{3,4,8,9\}$, so we have $R= \{z_i,i\in\{1,2,5,6,7,10\}\}$. The four real roots all satisfy $zq(z)-p(z)=0$, so we have $R'=R_1'$. By using formulas (\ref{specnrmSform}), (\ref{specnrmSform1}), (\ref{realforspecnrm}), (\ref{realforspecnrm1}), we get:
	\begin{equation*}
	\begin{aligned}
	\|\cS\|_{\sigma~~} &  = \max\{|s_4|, \lambda_q(z), z\in R \} = \max\{|s_4|, \lambda_v(z), z\in R_1 \} =  0.8872,\\
	\|\cS\|_{\sigma,\R} & = \max\{|s_4|,\lambda_q(z),z\in R'\} = \max\{|s_4|,\lambda_v(z),z\in R_1'\} =0.8872.
	\end{aligned}
	\end{equation*}
	
	\begin{table}\caption{Computational results for Example \ref{exm:3:d:qubit}.} \label{table:exm:3:d:qubit}
		\centering
		\begin{scriptsize}
			\begin{tabular}{|c|r|c|c|c||c|r|c|c|c|c|} \hline
				No. &    $z$\quad\quad\quad\quad  & $\F$ & $\lambda_q(z)$ & $\lambda_v(z)$  & No. &    $z$\quad\quad\quad\quad  & $\F$ & $\lambda_q(z)$ & $\lambda_v(z)$  \\ 
				\hline 
				1& 0.8217  & $\R$  & 0.3543 & 0.3543 & 6 & 0.1045    & $\R$   & 0.1632   &  0.1632  \\ \hline              
				2 & 0.8070+0.2388i & $\C$ & 0.3540&  0.3540 & 7 &-1.0376     &   $\R$ & 0.8872   &    0.8872              \\  \hline   
				3& 4.6051+0.9267i &  $\C$    & 0.2484 &  0.0838 &   8& 0.1990-0.3520i   &  $\C$ & 0.0639   &  0.1894            \\  \hline 
				4 & 0.1990+0.3520i  &  $\C$  & 0.0639  &   0.1894&    9 & 4.6051-0.9267i     & $\C$& 0.2484  & 0.0838             \\  \hline 
				5 & 5.6114  & $\R$  & 0.0270  & 0.0270   &  10 & 0.8070-0.2388i  &    $\C$& 0.3540 & 0.3540     \\  \hline   
			\end{tabular} 
		\end{scriptsize}
	\end{table} 
	
\end{exm}

\begin{exm}\emph{\cite[Example 6.1]{AMM10}} \label{exm:9:d:qubit} Given tensor $\cS\in\rS^4\R^2$ with: 
	$$\cS_{1,1,1,1} = \sqrt{\frac{1}{3}};\quad \cS_{1,2,2,2} = \frac{1}{2}\sqrt{\frac{2}{3}}.$$ 
	The polynomial $zv(z)-u(z)$ has degree 10. It 10 roots, 4 of them are real and the other 6 are complex.  The computational results are shown in Table \ref{table:exm:9:d:qubit}.  Clearly, $|s_{4}| = |\cS_{2,2,2,2}|=0$ and we have $R= R_1$, $R'=R_1'$. By using formulas (\ref{specnrmSform}), (\ref{specnrmSform1}), (\ref{realforspecnrm}), (\ref{realforspecnrm1}), we get:	
	\begin{equation*}
	\begin{aligned}
	\|\cS\|_{\sigma~~} &  = \max\{|s_4|, \lambda_q(z), z\in R \} = \max\{|s_4|, \lambda_v(z), z\in R_1 \} =  0.5774,\\
	\|\cS\|_{\sigma,\R} & = \max\{|s_4|,\lambda_q(z),z\in R'\} = \max\{|s_4|,\lambda_v(z),z\in R_1'\} =0.5774.
	\end{aligned}
	\end{equation*}
	According to \cite{AMM10}, $\|\cS\|_{\sigma}=\frac{1}{\sqrt{3}}$.
	
	\begin{table}\caption{Computational results for Example \ref{exm:9:d:qubit}.} \label{table:exm:9:d:qubit}
		\centering
		\begin{scriptsize}
			\begin{tabular}{|c|r|c|c|c||c|r|c|c|c|c|} \hline
				No. &  $z$\quad\quad\quad\quad   & $\F$ & $\lambda_q(z)$ & $\lambda_v(z)$ & No. &  $z$\quad\quad\quad\quad   & $\F$ & $\lambda_q(z)$ & $\lambda_v(z)$ \\ \hline 
				1& 0.5176           &  $\R$ & 0.5  &   0.5  &  6& 0 & $\R$&     0.5774& 0.5774 \\ \hline              
				2& 0.9659+1.6730i & $\C$  &      0.5 &   0.5&   7&-0.7071-1.2247i  & $\C$&   0.5774& 0.5774   \\ \hline            
				3&-0.2588+0.4483i & $\C$ &        0.5 &     0.5 &  8&-0.2588-0.4483i  & $\C$&       0.5 &   0.5    \\ \hline             
				4&-0.7071+1.2247i & $\C$ &   0.5774 & 0.5774  &   9& 0.9659-1.6730i   & $\C$&      0.5 &  0.5      \\ \hline          
				5&-1.9319   & $\R$ &      0.5 &    0.5    & 10& 1.4142     & $\R$&  0.5774 & 0.5774  \\ \hline           
			\end{tabular} 
		\end{scriptsize}
	\end{table} 
\end{exm}

\begin{exm} \label{exm:4:d:qubit} Given tensor $\cS\in\rS^5\R^2$ with:
	$$\cS_{1,1,1,1,1} = \frac{2}{\sqrt{574}};\cS_{1,1,1,1,2} =  \frac{-1}{\sqrt{574}};\cS_{1,1,1,2,2}= \frac{4}{\sqrt{574}}; $$
	$$\cS_{1,1,2,2,2} =  \frac{-6}{\sqrt{574}}; \cS_{1,2,2,2,2} =  \frac{-2}{\sqrt{574}}; \cS_{2,2,2,2,2} = \frac{5}{\sqrt{574}}. $$	
	The polynomial $zv(z)-u(z)$ has degree 17.   It has 17 roots, 9 of them are real and the other 8 are complex.  The computational results are shown in Table \ref{table:exm:4:d:qubit}.  Clearly $|s_{5}| = |\cS_{2,2,2,2,2}|=\frac{5}{\sqrt{574}}$. For the 17 roots $z$, (\ref{antifixqub1}) fails to hold for 12 roots, so we have $R= \{z_i,i\in\{1,6,9,10,16\}\}$. These five roots are also real, so we have $R'=R$. By using formulas (\ref{specnrmSform}), (\ref{specnrmSform1}), (\ref{realforspecnrm}), (\ref{realforspecnrm1}), we get:	
	\begin{equation*}
	\begin{aligned}
	\|\cS\|_{\sigma~~} & = \max\{|s_5|, \lambda_q(z), z\in R \} = \max\{|s_5|, \lambda_v(z), z\in R_1 \} =  0.6930,\\
	\|\cS\|_{\sigma,\R} & = \max\{|s_5|,\lambda_q(z),z\in R'\} = \max\{|s_5|,\lambda_v(z),z\in R_1'\} =0.6930.
	\end{aligned}
	\end{equation*}

	\begin{table}\caption{Computational results for Example \ref{exm:4:d:qubit}.} \label{table:exm:4:d:qubit}
		\centering
		\begin{scriptsize}
			\begin{tabular}{|c|r|c|c|c||c|r|c|c|c|} \hline
				No. &   $z$\quad\quad\quad\quad  & $\F$ & $\lambda_q(z)$ & $\lambda_v(z)$ & No. &  $z$\quad\quad\quad\quad  & $\F$ & $\lambda_q(z)$ & $\lambda_v(z)$\\ 
				\hline 
				1& 1.4857     & $\R$ & 0.2864  &    0.2864 &  10& 0.0928    & $\R$  &   0.0748  &    0.0748      \\  \hline      
				2& 0.3050                & $\R$ &   0.0845  &   0.0809& 11& -1.3611 -    0.6846i & $\C$    & 0.9249  &    0.5834  \\  \hline            
				3& 2.8803 +    0.6699i    &$\C$  & 0.7952  &    0.2302    & 12& -0.4087 -    0.4279i  & $\C$    & 0.3564  &    0.5650  \\  \hline        
				4& 0.1503     & $\R$  & 0.0750  &    0.0784     &  13& -3.7577   & $\R$   &  0.8736  &    0.2489  \\  \hline     
				5& 0.1379 +    0.2586i   & $\C$   & 0.0434  &   0.1501  & 14& 0.1379 -    0.2586i  & $\C$    & 0.0434  &    0.1501   \\  \hline         
				6& -15.466               & $\R$     &  0.2224  &    0.2224   & 15& 2.8803 -    0.6699i  & $\C$     & 0.7952  &    0.2302   \\  \hline          
				7& -0.4087 +    0.4279i  & $\C$    &  0.3564  &    0.5650  &  16& 0.2667                & $\R$    & 0.0819  &   0.0819  \\  \hline          
				8& -1.3611 +    0.6846i  & $\C$    & 0.9249  &    0.5834  & 17& 0.6923     & $\R$   & 0.0536  &    0.1882         \\  \hline
				9& -0.8795    & $\R$  &  0.6930  &    0.6930 &   &&&&    \\  \hline  
			\end{tabular} 
		\end{scriptsize}
	\end{table}  
	
\end{exm}

\begin{exm}\emph{\cite[Example 6.2(b)]{AMM10}} \label{exm:10:d:qubit} Given tensor $\cS\in\rS^5\R^2$ with:
	$$\cS_{1,1,1,1,1} = \frac{1}{\sqrt{1+A^2}}; \cS_{1,2,2,2,2} = \frac{A}{\sqrt{5(1+A^2)}}; A \approx 1.53154.$$  
	The polynomial $zv(z)-u(z)$ has degree 17.  It has 17 roots, 5 of them are real and the other 12 are complex. The computational results are shown in Table \ref{table:exm:10:d:qubit}. Clearly $|s_{5}| = |\cS_{2,2,2,2,2}|=0$. For the 17 roots $z$, (\ref{antifixqub1}) fails to hold for 4 roots $z_i,i\in\{2,8,12,16\}$, so we have $R=\{z_i,i\in[17]\backslash\{2,8,12,16\}\}$.  Furthermore $R'=R_1'$. By using formulas (\ref{specnrmSform}), (\ref{specnrmSform1}), (\ref{realforspecnrm}), (\ref{realforspecnrm1}), we get:	
	\begin{equation*}
	\begin{aligned}
	\|\cS\|_{\sigma~~} & = \max\{|s_5|, \lambda_q(z), z\in R \} = \max\{|s_5|, \lambda_v(z), z\in R_1 \} =  0.5492,\\
	\|\cS\|_{\sigma,\R} & = \max\{|s_5|,\lambda_q(z),z\in R'\}= \max\{|s_5|,\lambda_v(z),z\in R_1'\} =0.5492.
	\end{aligned}
	\end{equation*}
	According to \cite{AMM10}, $\|\cS\|_{\sigma}=\frac{1}{\sqrt{1+A^2}}$.

	\begin{table}\caption{Computational results for Example \ref{exm:10:d:qubit}.} \label{table:exm:10:d:qubit}
		\centering
		\begin{scriptsize}
			\begin{tabular}{|c|r|c|c|c||c|r|c|c|c|c|} \hline
				No. &  $z$\quad\quad\quad\quad   & $\F$ & $\lambda_q(z)$ & $\lambda_v(z)$ & No. &  $z$\quad\quad\quad\quad   & $\F$ & $\lambda_q(z)$ & $\lambda_v(z)$ \\ \hline 
				1& 0.6398 & $\R$& 0.3657 &  0.3657  & 10&-0.6398 &$\R$&  0.3657 & 0.3657      \\ \hline           
				2& 0.4115 +    0.4115i &$\C$&  0.3269 &  0.3269   &   11&-1.4729 -     1.4729i   &$\C$&     0.5258 & 0.5258    \\ \hline         
				3& 1.4729 +     1.4729i &$\C$&  0.5258 &  0.5258    & 12&-0.4115 -    0.4115i &$\C$&       0.3269&  0.3269 \\ \hline           
				4& 0 &$\R$&  0.5492& 0.5492 &    13&  -     1.8949i  &$\C$&     0.5457 & 0.5457   \\ \hline          
				5&    1.8949i &$\C$&  0.5457 & 0.5457           &  14&   -    0.6398i &$\C$&      0.3657 & 0.3657  \\ \hline  
				6&    0.6398i&$\C$&   0.3657 & 0.3657  & 15& 1.4729 -     1.4729i   &$\C$&     0.5258 &  0.5258   \\ \hline             
				7&-1.4729 +     1.4729i  &$\C$&  0.5258&  0.5258      &  16& 0.4115 -    0.4115i  &$\C$&      0.3269&  0.3269  \\ \hline       
				8&-0.4115 +    0.4115i &$\C$&  0.3269&  0.3269      &   17& 1.8949   &$\R$&     0.5457& 0.5457  \\ \hline     
				9&-1.8949 &$\R$&   0.5457& 0.5457  &&&&& \\ \hline   
			\end{tabular} 
		\end{scriptsize}
	\end{table}
\end{exm}

\begin{exm} \label{exm:5:d:qubit} Given tensor $\cS\in\rS^6\R^2$ with:
	$$\cS_{1,1,1,1,1,1} = \frac{1}{\sqrt{641}};\cS_{1,1,1,1,1,2} =  \frac{-1}{\sqrt{641}};\cS_{1,1,1,1,2,2}= \frac{2}{\sqrt{641}};$$ 
	$$\cS_{1,1,1,2,2,2} =  \frac{-3}{\sqrt{641}}; \cS_{1,1,2,2,2,2} = \frac{4}{\sqrt{641}}; \cS_{1,2,2,2,2,2} = \frac{5}{\sqrt{641}}; \cS_{2,2,2,2,2,2} =  \frac{-2}{\sqrt{641}}.$$  	
	The polynomial $zv(z)-u(z)$ has degree 26. It has 26 roots, 6 of them are real and the other 20 are complex.  The computational results are shown in Table \ref{table:exm:5:d:qubit}.  Clearly $|s_{6}| = |\cS_{2,2,2,2,2,2}|=\frac{2}{\sqrt{641}}$. For the 26 roots $z$, (\ref{antifixqub1}) fails to hold for 20 roots, so we have $R= \{z_i,i\in \{1,7,8,14,21,26\}\}$. 4 of the 6 real roots satisfy $zq(z)-p(z)=0$, so we have $R'=\{z_i,i\in\{ 1,8,14,26\}\}$. By using formulas (\ref{specnrmSform}), (\ref{specnrmSform1}), (\ref{realforspecnrm}), (\ref{realforspecnrm1}), we get:
	\begin{equation*}
	\begin{aligned}
	\|\cS\|_{\sigma~~} &  = \max\{|s_6|, \lambda_q(z), z\in R \} = \max\{|s_6|, \lambda_v(z), z\in R_1 \} =  0.6336,\\
	\|\cS\|_{\sigma,\R} & = \max\{|s_6|,\lambda_q(z),z\in R'\} = \max\{|s_6|,\lambda_v(z),z\in R_1'\} =0.6336.
	\end{aligned}
	\end{equation*}

	\begin{table}\caption{Computational results for Example \ref{exm:5:d:qubit}.} \label{table:exm:5:d:qubit}
		\centering
		\begin{scriptsize}
			\begin{tabular}{|c|r|c|c|c||c|r|c|c|c|c|} \hline
				No. &  $z$\quad\quad\quad\quad   & $\F$ & $\lambda_q(z)$ & $\lambda_v(z)$ & No. &  $z$\quad\quad\quad\quad   & $\F$ & $\lambda_q(z)$ & $\lambda_v(z)$ \\  
				\hline 
				1 & 1.8997     & $\R$          & 0.4503  &    0.4503 &  14& -0.8764   & $\R$&     0.6336  &    0.6336   \\  \hline           
				2 & 0.7256   & $\R$  &     0.0507  &    0.1951 &  15&-2.3660 -    0.5050i &  $\C$  &  0.7105  &    0.2176 \\  \hline             
				3& 0.2341 +    0.2263i  & $\C$  &  0.0136  &    0.0136   & 16& 0.1920 -    0.0393i  & $\C$ &   0.0213  &   0.0201     \\  \hline           
				4& 0.2284 +    0.2369i  & $\C$ &  0.0130  &    0.0138   & 17&-1.3863 -    0.7210i &  $\C$ &   0.8063  &    0.5087   \\  \hline            
				5& 7.5466 +     4.4315i &$\C$    &  2.3011  &    0.1821 & 18&-0.3753 -    0.5248i & $\C$ &     0.3004  &     0.4761  \\  \hline             
				6& 0.2460 +    0.3647i  & $\C$ &  0.0193  &    0.0631   &   19& 0.1379 -    0.2571i  & $\C$  & 0.0029  &    0.0365   \\  \hline  
				7& 0.2179 +    0.2468i  & $\C$ &  0.0121  &    0.0121  & 20& 7.5466 -     4.4315i & $\C$   &  2.3011  &    0.1821   \\  \hline            
				8& -6.0779              & $\R$   &    0.1831  &    0.1831 & 21& 0.2179 -    0.2468i  &  $\C$&  0.0121  &    0.0121  \\  \hline                
				9& 0.1379 +    0.2571i  &  $\C$ & 0.0029  &    0.0365   &  22& 0.2460 -    0.3647i  &  $\C$&  0.0193  &    0.0631  \\  \hline        
				10&-0.3753 +    0.5248i & $\C$ &     0.3004  &    0.4761  &  23& 4.1026     & $\R$&  1.2078  &    0.3138            \\  \hline         
				11&-1.3863 +    0.7210i &$\C$   &   0.8063  &    0.5087   &24& 0.2284 -    0.2369i  & $\C$ &  0.0130  &   0.0138              \\  \hline          
				12& 0.1920 +    0.0393i  & $\C$ &  0.0213  &    0.0201     &  25& 0.2341 -    0.2263i  &  $\C$& 0.0136  &   0.0136               \\  \hline       
				13& -2.366 +    0.5050i  & $\C$  &  0.7105  &   0.2176  &  26& 0.1865     &$\R$ &  0.0219  &   0.0219              \\  \hline 
			\end{tabular} 
		\end{scriptsize}
	\end{table} 
	
\end{exm}

\begin{exm}\emph{\cite[Example 6.3]{AMM10}} \label{exm:11:d:qubit} Given tensor $\cS\in\rS^6\R^2$ with:
	$$\cS_{1,1,1,1,1,2} = \frac{1}{2\sqrt{3}};\cS_{1,2,2,2,2,2} = \frac{1}{2\sqrt{3}}.$$ 
	The polynomial $zv(z)-u(z)$ has degree 25.  It has 25 roots, 7 of them are real and the other 18 are complex. The computational results are shown in Table \ref{table:exm:11:d:qubit}. Clearly $|s_{6}| = |\cS_{2,2,2,2,2,2}|=0$. 	For the 25 roots $z$, (\ref{antifixqub1}) fails to hold for 5 roots $z_i,i\in \{4,10,12,16,22\}$, so we have $R= \{z_i,i\in [25]\backslash \{4,10,12,16,22\}\}$. 6 of the 7 real roots satisfy $zq(z)-p(z)=0$, so we have 
	$R'=R\cap \R=\{z_i,i\in \{1,2,13,14,15,25\}\}$. By using formulas (\ref{specnrmSform}), (\ref{specnrmSform1}), (\ref{realforspecnrm}), (\ref{realforspecnrm1}), we get:
	\begin{equation*}
	\begin{aligned}
	\|\cS\|_{\sigma~~} &  = \max\{|s_6|, \lambda_q(z), z\in R\} = \max\{|s_6|, \lambda_v(z), z\in R_1 \} = 0.4714,\\
	\|\cS\|_{\sigma,\R} & = \max\{|s_6|,\lambda_q(z),z\in R'\} = \max\{|s_6|,\lambda_v(z),z\in R_1'\} =0.4714.
	\end{aligned}
	\end{equation*}
	According to \cite{AMM10}, $\|\cS\|_{\sigma}=\frac{\sqrt{2}}{3}$.
	
	\begin{table}\caption{Computational results for Example \ref{exm:11:d:qubit}.} \label{table:exm:11:d:qubit}
		\centering
		\begin{scriptsize}
			\begin{tabular}{|c|r|c|c|c||c|r|c|c|c|c|} \hline
				No. &  $z$\quad\quad\quad\quad   & $\F$ & $\lambda_q(z)$ & $\lambda_v(z)$ & No. &  $z$\quad\quad\quad\quad   & $\F$ & $\lambda_q(z)$ & $\lambda_v(z)$ \\ \hline  
				1& 1    &$\R$ & 0.4330 & 0.4330 &  14& -1.9319      &$\R$& 0.4714  &  0.4714   \\ \hline             
				2& 1.9319 &$\R$& 0.4714 & 0.4714   &   15& -0.5176   &$\R$&      0.4714 & 0.4714\\ \hline          
				3& 0.2929 +    0.2929i &$\C$&     0.4330& 0.4330 &  16& -0.7071 -  0.7071i &$\C$&     0.2887 & 0.2887 \\ \hline           
				4& 0.7071 +    0.7071i  &$\C$&    0.2887& 0.2887 & 17& -1.7071 -     1.7071i  &$\C$&   0.4330& 0.4330\\ \hline             
				5& 1.7071 +     1.7071i    &$\C$&  0.4330& 0.4330  & 18& -0.2929 -    0.2929i &$\C$&     0.4330 & 0.4330  \\ \hline             
				6&   0.5176i &$\C$&      0.4714&  0.4714 &   19&   -  1i&$\C$&      0.4330  & 0.4330  \\ \hline           
				7&   1i  &$\C$&    0.4330 & 0.4330 & 20&  -     1.9319i   &$\C$&    0.4714   & 0.4714  \\ \hline            
				8&  1.9319i     &$\C$&  0.4714 &  0.4714     & 21&  -    0.5176i    &$\C$&   0.4714&  0.4714     \\ \hline          
				9& -0.2929 +    0.2929i &$\C$&     0.4330 & 0.4330&   22& 0.7071 -    0.7071i   &$\C$&   0.2887& 0.2887 \\ \hline           
				10& -0.7071 +    0.7071i  &$\C$&    0.2887 & 0.2887& 23& 0.2929 -    0.2929i  &$\C$&    0.4330& 0.4330 \\ \hline              
				11& -1.7071 +     1.7071i    &$\C$&  0.4330& 0.4330  & 24& 1.7071 -     1.7071i &$\C$&     0.4330 & 0.4330 \\ \hline            
				12& 0 &$\R$& 0 & 0.2887& 25& 0.5176    &$\R$&     0.4714 &  0.4714 	\\ \hline             
				13& -1         &$\R$& 0.4330  & 0.4330         &   &&&&   \\ \hline 
			\end{tabular} 
		\end{scriptsize}
	\end{table}  
\end{exm}

\begin{exm} \label{exm:6:2:d:qubit} Given tensor $\cS\in\rS^7\R^2$ with:
	$$\cS_{1,1,1,1,1,1,1} = \frac{1}{\sqrt{37}};\cS_{1,1,1,1,1,1,2} =  \frac{-1}{\sqrt{74}};\cS_{1,1,1,1,1,2,2}=  \frac{1}{\sqrt{111}};\cS_{1,1,1,1,2,2,2} = \frac{1}{\sqrt{185}}; $$ 
	$$\cS_{1,1,1,2,2,2,2} = \frac{1}{\sqrt{185}}; \cS_{1,1,2,2,2,2,2} = \frac{1}{\sqrt{111}}; \cS_{1,2,2,2,2,2,2} =\frac{1}{\sqrt{74}}; \cS_{2,2,2,2,2,2,2} = \frac{1}{\sqrt{37}};$$  	   
	The polynomial $zv(z)-u(z)$ has degree 37.  It has 37 roots, 3 of them are real and the other 34 are complex.  The computational results are shown in Table \ref{table:exm:6:2:d:qubit}.  Clearly $|s_{7}| = |\cS_{2,2,2,2,2,2,2}|=\frac{1}{\sqrt{37}}$. For the 37 roots $z$, (\ref{antifixqub1}) fails to hold for 28 roots, so we have $R= \{z_i,i\in \{1,2,7,13,17,21,23,30,32\}\}$.  3 real roots all satisfy $zq(z)-p(z)=0$, so we have $R'=R\cap \R=\{z_i,i\in \{1,2,21\}\}$. By using formulas (\ref{specnrmSform}), (\ref{specnrmSform1}), (\ref{realforspecnrm}), (\ref{realforspecnrm1}), we get:	
	\begin{equation*}
	\begin{aligned}
	\|\cS\|_{\sigma~~} &  = \max\{|s_7|, \lambda_q(z), z\in R \} = \max\{|s_7|, \lambda_v(z), z\in R_1 \} = 0.8741,\\
	\|\cS\|_{\sigma,\R} & = \max\{|s_7|,\lambda_q(z),z\in R'\} = \max\{|s_7|,\lambda_v(z),z\in R_1'\} = 0.8741.
	\end{aligned}
	\end{equation*}

	\begin{table}\caption{Computational results for Example \ref{exm:6:2:d:qubit}.} \label{table:exm:6:2:d:qubit}
		\centering
		\begin{scriptsize}
			\begin{tabular}{|c|r|c|c|c||c|r|c|c|c|c|} \hline
				No. &  $z$\quad\quad\quad\quad   & $\F$ & $\lambda_q(z)$ & $\lambda_v(z)$ & No. &  $z$\quad\quad\quad\quad   & $\F$ & $\lambda_q(z)$ & $\lambda_v(z)$ \\ \hline 
				1& 1.2348    &$\R$&    0.8741  & 0.8741 & 20&-1.4583 -    0.2619i &$\C$&     0.1131  & 0.0750   \\ \hline               
				2& 0.1653    &$\R$&    0.0893 & 0.0893  & 21&-0.3919    &$\R$&     0.4121   &   0.4121         \\ \hline             
				3&-2.2667 +     1.2252i   &$\C$&   0.0404 & 0.0310 & 22&-2.0392 -    0.5786i &$\C$&    0.0292  & 0.0269    \\ \hline              
				4&0.1442 +    0.1391i  &$\C$&      0.0581&  0.1150 &  23&-2.1260 -    0.7952i  &$\C$&    0.0218 &  0.0218   \\ \hline             
				5&0.1046 +    0.2732i   &$\C$&    0.1120& 0.1851   & 24&-1.5868 -     1.4793i &$\C$&     0.0982  & 0.0496  \\ \hline           
				6&-1.5868 +     1.4793i   &$\C$&    0.0982 & 0.0496&  25& 0.0713 -    0.1854i &$\C$&     0.1062 &  0.1601   \\ \hline              
				7&0.0208  +    0.5582i  &$\C$&     0.3209 & 0.3209 &  26&-1.1085 -     1.4166i  &$\C$&    0.1252  & 0.0758  \\ \hline           
				8&-0.2315 +     1.5252i  &$\C$&     0.0482 &  0.0629& 27&-0.3327 -     1.5213i &$\C$&    0.0533&  0.0579     \\ \hline              
				9&-0.0557 +     1.4844i   &$\C$&   0.0955 & 0.0955  & 28&-0.2315 -     1.5252i &$\C$&     0.0482  &  0.0629  \\ \hline            
				10&0.0368 +     1.0175i   &$\C$&    0.2890  & 0.1871& 29&0.0368 -     1.0175i &$\C$&     0.2890&   0.1871    \\ \hline             
				11&0.0713 +    0.1854i   &$\C$&    0.1062 &  0.1601 &  30&0.0208  -    0.5582i &$\C$&     0.3209  &  0.3209  \\ \hline             
				12&-0.3327 +     1.5213i  &$\C$&    0.0533  & 0.0579& 31&-0.0557 -     1.4844i &$\C$&    0.0955   & 0.0955  \\ \hline                
				13&-0.4844 +     1.6068i  &$\C$&     0.0652&  0.0652& 32&-0.4844 -     1.6068i  &$\C$&    0.0652   &  0.0652 \\ \hline               
				14&-1.1085 +     1.4166i  &$\C$&     0.1252  &  0.0759 & 33&0.1046 -    0.2732i  &$\C$&    0.1120  &  0.1851    \\ \hline               
				15&-2.1908 +    0.9311i  &$\C$&    0.0224   & 0.0224   &  34&0.1442 -    0.1391i  &$\C$&     0.0581 &  0.1150   \\ \hline           
				16&0.1701 +    0.0760i   &$\C$&   0.0750     & 0.1158  &  35&-2.2667 -     1.2252i &$\C$&    0.0404   & 0.0310  \\ \hline           
				17&-2.1260 +    0.7952i   &$\C$&    0.0218   & 0.0218  & 36&0.1701 -   0.0760i  &$\C$&   0.0750   &   0.1158   \\ \hline            
				18&-1.4583 +    0.2619i &$\C$&      0.1131  & 0.0750   &  37&-2.1908 -    0.9311i  &$\C$&   0.0224    & 0.0224   \\ \hline  
				19&-2.0392 +    0.5786i &$\C$&    0.0292  & 0.0269   & &&&& \\ \hline  
			\end{tabular} 
		\end{scriptsize}
	\end{table}   
	
\end{exm}

\begin{exm}\emph{\cite[Example 6.4]{AMM10}} \label{exm:12:d:qubit} Given tensor $\cS\in\rS^7\R^2$ with:
	$$ \cS_{1,1,1,1,1,1,2} = \frac{1}{\sqrt{14}};\cS_{1,2,2,2,2,2,2} = \frac{1}{\sqrt{14}}.$$ 	
	The polynomial $zv(z)-u(z)$ has degree 36.  It has 36 roots, 6 of them are real and the other 30 are complex.  The computational results are shown in Table \ref{table:exm:12:d:qubit}.  Clearly $|s_{7}| = |\cS_{2,2,2,2,2,2,2}|=0$. For the 36 roots $z$, (\ref{antifixqub1}) fails to hold for 11 roots $z_i,i\in I$, so we have $R= \{z_i,i\in [36]\backslash I\}$, where $I=\{3,4,5,11,12,18,20,25,27,33,34\}$. 5 of the 6 real roots satisfy $zq(z)-p(z)=0$, so we have 
	$R'=R\cap \R=\{z_i,i\in \{1,2,19,21,36\}\}$. By using formulas (\ref{specnrmSform}), (\ref{specnrmSform1}), (\ref{realforspecnrm}), (\ref{realforspecnrm1}), we get:	
	\begin{equation*}
	\begin{aligned}
	\|\cS\|_{\sigma~~} &  = \max\{|s_7|, \lambda_q(z), z\in R \} = \max\{|s_7|, \lambda_v(z), z\in R_1 \} = 0.4508,\\
	\|\cS\|_{\sigma,\R} & = \max\{|s_7|,\lambda_q(z),z\in R'\} = \max\{|s_7|,\lambda_v(z),z\in R_1'\} =0.4508.
	\end{aligned}
	\end{equation*}
	
	\begin{table}\caption{Computational results for Example \ref{exm:12:d:qubit}.} \label{table:exm:12:d:qubit}
		\centering
		\begin{scriptsize}
			\begin{tabular}{|c|r|c|c|c||c|r|c|c|c|c|} \hline
				No. &  $z$\quad\quad\quad\quad   & $\F$ & $\lambda_q(z)$ & $\lambda_v(z)$ & No. &  $z$\quad\quad\quad\quad   & $\F$ & $\lambda_q(z)$ & $\lambda_v(z)$ \\ \hline  
				1& 1     &$\R$&    0.3307 & 0.3307&  19&-0.3964   &$\R$&     0.4406 & 0.4406 \\ \hline             
				2& 2.3570           &$\R$&  0.4508& 0.4508    & 20&-0.9914 -    0.1312i &$\C$&      0.2478& 0.2478 \\ \hline          
				3& 0 &$\R$&  0  &  0.2673 &  21&-2.5227                &$\R$&     0.4406& 0.4406  \\ \hline            
				4& 0.8791 +    0.4766i &$\C$&      0.2478 & 0.2478 & 22&-0.3432 -    0.2494i &$\C$&      0.4508& 0.4508 \\ \hline             
				5& 0.7249 +    0.6888i &$\C$&      0.2478 & 0.2478 & 23&-0.8090 -    0.5878i &$\C$&      0.3307& 0.3307 \\ \hline             
				6& 2.0409 +     1.4828i  &$\C$&     0.4406 & 0.4406  & 24&-1.9069 -     1.3854i   &$\C$&    0.4508& 0.4508  \\ \hline            
				7& 0.3207 +      0.2330i &$\C$&      0.4406& 0.4406   & 25&-0.4311 -    0.9023i   &$\C$&   0.2478&  0.2478  \\ \hline            
				8& 0.3090 +    0.9511i&$\C$&       0.3307& 0.3307  & 26&-0.1225 -      0.3770i   &$\C$&   0.4406 & 0.4406    \\ \hline           
				9& 0.1311 +     0.4035i &$\C$&      0.4508& 0.4508  &  27&-0.1816 -    0.9834i   &$\C$&    0.2478& 0.2478  \\ \hline          
				10& 0.7284 +     2.2417i&$\C$&       0.4508& 0.4508 & 28&-0.7796 -     2.3992i  &$\C$&     0.4406& 0.4406  \\ \hline             
				11&-0.1816 +    0.9834i&$\C$&       0.2478& 0.2478  & 29& 0.1311 -     0.4035i   &$\C$&    0.4508 & 0.4508 \\ \hline            
				12&-0.4311 +    0.9023i&$\C$&       0.2478& 0.2478  &  30& 0.3090 -    0.9511i   &$\C$&    0.3307& 0.3307 \\ \hline           
				13&-0.1225 +      0.3770i&$\C$&       0.4406 & 0.4406 &  31& 0.7284 -     2.2417i   &$\C$&    0.4508& 0.4508  \\ \hline           
				14&-0.7796 +     2.3992i&$\C$&       0.4406& 0.4406 &  32& 0.3207 -      0.2330i   &$\C$&    0.4406 & 0.4406 \\ \hline           
				15&-0.8090 +    0.5878i&$\C$&       0.3307& 0.3307 & 33& 0.7249 -    0.6888i  &$\C$&     0.2478& 0.2478 \\ \hline              
				16&-0.3432 +    0.2494i&$\C$&       0.4508& 0.4508 & 34& 0.8791 -    0.4766i &$\C$&      0.2478& 0.2478 \\ \hline            
				17&-1.9069 +     1.3854i &$\C$&      0.4508& 0.4508  & 35& 2.0409 -     1.4828i  &$\C$&    0.4406 & 0.4406  \\ \hline            
				18&-0.9914 +    0.1312i &$\C$&      0.2478& 0.2478 &  36& 0.4243  &$\R$&    0.4508 & 0.4508	\\ \hline   
			\end{tabular} 
		\end{scriptsize}
	\end{table} 
\end{exm}

\begin{exm}
	\label{exm:7:d:qubit} Given tensor $\cS\in\rS^8\R^2$ with:
	\begin{equation*}
	\begin{aligned}
	\cS_{1,1,1,1,1,1,1,1} & = \frac{1}{\sqrt{52}};\cS_{1,1,1,1,1,1,1,2} = \frac{1}{2\sqrt{26}};\cS_{1,1,1,1,1,1,2,2}= \frac{1}{\sqrt{182}}; \cS_{1,1,1,1,1,2,2,2}   =  \frac{-1}{2\sqrt{52}};\\
	\cS_{1,1,1,2,2,2,2,2} &  =  \frac{-1}{2\sqrt{52}}; \cS_{1,1,2,2,2,2,2,2}  =  \frac{1}{\sqrt{182}};\cS_{1,2,2,2,2,2,2,2} = \frac{-1}{2\sqrt{52}};\cS_{2,2,2,2,2,2,2,2}  = \frac{1}{\sqrt{52}}. 
	\end{aligned}
	\end{equation*}	 
	The polynomial $zv(z)-u(z)$ has degree 50.  It has 50 roots, 6 of them are real and the other 44 are complex.  The computational results are shown in Table \ref{table:exm:7:d:qubit}.  Clearly $|s_{8}| = |\cS_{2,2,2,2,2,2,2,2}|=\frac{1}{\sqrt{52}}$. For the 50 roots $z$, (\ref{antifixqub1}) fails to hold for 38 roots, so we have $R= \{z_i,i\in \{1,7,10,13,16,26,32,33,42,44,45,46\}\}$. 4 of the 6 real roots satisfy $zq(z)-p(z)=0$, so we have $R'=R\cap \R=\{z_i,i\in\{1,26,44,46\}\}$. By using formulas (\ref{specnrmSform}), (\ref{specnrmSform1}), (\ref{realforspecnrm}), (\ref{realforspecnrm1}), we get: 
	\begin{equation*}
	\begin{aligned}
	\|\cS\|_{\sigma~~} &  = \max\{|s_8|, \lambda_q(z), z\in R \} = \max\{|s_8|, \lambda_v(z), z\in R_1 \} = 0.7713,\\
	\|\cS\|_{\sigma,\R} & = \max\{|s_8|,\lambda_q(z),z\in R'\} = \max\{|s_8|,\lambda_v(z),z\in R_1'\} = 0.7713.
	\end{aligned}
	\end{equation*}

	\begin{table}\caption{Computational results for Example \ref{exm:7:d:qubit}.} \label{table:exm:7:d:qubit}
		\centering
		\begin{scriptsize}
			\begin{tabular}{|c|r|c|c|c||c|r|c|c|c|c|} \hline
				No. &  $z$\quad\quad\quad\quad   & $\F$ & $\lambda_q(z)$ & $\lambda_v(z)$ & No. &  $z$\quad\quad\quad\quad   & $\F$ & $\lambda_q(z)$ & $\lambda_v(z)$ \\ \hline   
				1& 1.1587    &$\R$&    0.2125 &  0.2125 &  26&-1.2070  &$\R$&     0.7713  & 0.7713  \\ \hline            
				2& 0.8721 +   0.0819i &$\C$&    0.0781  &  0.1757  &  27&-0.1650 -    0.0499i &$\C$&    0.0658  & 0.0959  \\ \hline              
				3& 0.6710 +  0.2000i &$\C$&     0.2161 &  0.2474  &  28& 3.2498 -    3.2320i  &$\C$&   0.0732& 0.0732    \\ \hline           
				4&2.4138 +    0.1986i &$\C$&     0.2289 & 0.0944   & 29&-0.1584 -    0.1218i &$\C$&    0.0342 &  0.1110  \\ \hline           
				5&0.4552  &$\R$&     0.2959 &  0.2384 &             30& 1.6006 -    4.6960i &$\C$&     0.2977 &  0.1037  \\ \hline
				6&3.0067 +     1.0869i  &$\C$&    0.1502 &  0.0822  &  31&-0.1569 -    0.2446i &$\C$&    0.0723 &   0.1837\\ \hline  
				7&3.2133 +     2.3777i &$\C$&    0.0707 & 0.0707   & 32&-0.2426 -     1.5964i &$\C$&     0.2144&   0.2144  \\ \hline            
				8&-0.0865   &$\R$&     0.0905 &  0.1123 & 33&-0.2436 -    0.5172i &$\C$&     0.3056 &   0.3056  \\ \hline              
				9&1.9392 +     2.5028i  &$\C$&    0.2103 & 0.0935    &  34&-0.5525 -    0.9157i &$\C$&     0.2502 &   0.2502 \\ \hline          
				10&0.3308 +    0.5789i  &$\C$&    0.4470 &  0.4470 &   35&-0.3519 -    0.7183i &$\C$&     0.3300 &   0.2266\\ \hline           
				11&0.3779 +     2.5267i &$\C$&     0.3744 &   0.1474 &  36&-0.4547 -    1.1690i&$\C$&     0.0638 &  0.1831  \\ \hline           
				12&-0.1053 +    0.1644i &$\C$&    0.0551  &  0.1338 &   37&-0.3127 -     1.2324i &$\C$&    0.0953  &  0.1742 \\ \hline            
				13&-0.2426 +     1.5964i  &$\C$&    0.2144 &  0.2144  &  38&-0.1312 -    0.9707i&$\C$&      0.2806  &  0.2451 \\ \hline  
				14&-0.1312 +    0.9707i &$\C$&     0.2806 &   0.2451 &   39&-0.1053 -    0.1644i &$\C$&    0.0551 &   0.1338 \\ \hline           
				15&-0.3127 +     1.2324i &$\C$&    0.0953 &  0.1742  &  40& 0.3779 -     2.5267i &$\C$&     0.3744 &   0.1474  \\ \hline            
				16&-0.2436 +    0.5172i &$\C$&     0.3056 & 0.3056  &  41& 0.5974 -     3.1012i&$\C$&      0.3610  &   0.1111 \\ \hline            
				17&-0.4547 +    1.1690i &$\C$&    0.0638 &  0.1831  &  42& 0.3308 -    0.5789i&$\C$&     0.4470 &   0.4470  \\ \hline          
				18&-0.5525 +    0.9157i &$\C$&     0.2502& 0.2502   & 43& 1.9392 -     2.5028i  &$\C$&    0.2103  & 0.0935   \\ \hline           
				19&-0.3519 +    0.7183i &$\C$&     0.3300&  0.2266   &  44&-0.1478  &$\R$&     0.0742  & 0.0742 \\ \hline          
				20&0.5974 +     3.1012i  &$\C$&    0.3610 &  0.1111   & 45& 3.2133 -     2.3777i &$\C$&    0.0707  & 0.0707    \\ \hline            
				21&-0.1569 +    0.2446i&$\C$&     0.0723 &  0.1837  & 46& 0.2848   &$\R$&     0.3181 &  0.3181   \\ \hline            
				22&1.6006 +     4.6960i &$\C$&     0.2977 &  0.1037   & 47& 3.0067 -     1.0869i  &$\C$&    0.1502  &  0.0822  \\ \hline            
				23&-0.1584 +    0.1218i &$\C$&    0.0342  &  0.1110  &  48& 0.6710  -    0.2000i&$\C$&      0.2161 &  0.2474  \\ \hline          
				24&3.2498 +      3.2320i  &$\C$&   0.0732 & 0.0732   &  49& 2.4138 -    0.1986i  &$\C$&    0.2289  & 0.0944   \\ \hline          
				25&-0.1650 +    0.0499i &$\C$&    0.0658  & 0.0959 & 50 & 0.8721 -   0.0819i &$\C$&    0.0781 &  0.1757  \\ \hline  
			\end{tabular} 
		\end{scriptsize}
	\end{table} 
	
\end{exm}

\begin{exm} \emph{\cite[Example 6.5]{AMM10}}
	\label{exm:13:d:qubit} Given tensor $\cS\in\rS^8\R^2$ with:
	$$ \cS_{1,1,1,1,1,1,1,2} = \frac{0.336}{\sqrt{2}}; \cS_{1,1,2,2,2,2,2,2} = \frac{0.741}{2\sqrt{7}}.$$  
	The polynomial $zv(z)-u(z)$ has degree 42, which has 41 roots, 7 of them are real and the other 34 are complex. One of the real roots $z=0$ has multiplicity 2.  The computational results are shown in Table \ref{table:exm:13:d:qubit}.
	Clearly, $|s_{8}| = |\cS_{2,2,2,2,2,2,2,2}|=0$. For the 41 roots $z$, (\ref{antifixqub1}) fails to hold for 10 roots, so we have $R=\{z_i,~i\in[41]\backslash I\}$, where
	$I=\{ 5,6,11,13,21,22,28,30,36,39\}$ 
	and $R'= R\cap\R$. By using formulas (\ref{specnrmSform}), (\ref{specnrmSform1}), (\ref{realforspecnrm}), (\ref{realforspecnrm1}), we get: 
	\begin{equation*}
	\begin{aligned}
	\|\cS\|_{\sigma~~} &  = \max\{|s_8|, \lambda_q(z), z\in R \} = \max\{|s_8|, \lambda_v(z), z\in R_1 \} = 0.4288,\\
	\|\cS\|_{\sigma,\R} & = \max\{|s_8|,\lambda_q(z),z\in R'\} = \max\{|s_8|,\lambda_v(z),z\in R_1'\} =0.4288.
	\end{aligned}
	\end{equation*}
	
	\begin{table}\caption{Computational results for Example \ref{exm:13:d:qubit}.} \label{table:exm:13:d:qubit}
		\centering
		\begin{scriptsize}
			\begin{tabular}{|c|r|c|c|c||c|r|c|c|c|c|} \hline
				No. &  $z$\quad\quad\quad\quad   & $\F$ & $\lambda_q(z)$ & $\lambda_v(z)$ & No. &  $z$\quad\quad\quad\quad   & $\F$ & $\lambda_q(z)$ & $\lambda_v(z)$ \\ \hline 
				1  &   0.8378            & $\R$& 0.3514 & 0.3514  & 22   &   -0.8187 -   0.0937i&$\C$&     0.2375 & 0.2375 \\ \hline               
				2  &   1.6193                &$\R$&    0.4286 & 0.4286 & 23   &   -1.8094               &$\R$&     0.4021 & 0.4021 \\ \hline            
				3  &   0.2941 +    0.2137i &$\C$&     0.4151& 0.4151 &  24  &   -0.3236 -    0.2351i &$\C$&     0.4288& 0.4288  \\ \hline             
				4  &   0.8992 +    0.6533i&$\C$&   1.3e-16 & 0.0040&  25   &   -0.6778 -    0.4924i &$\C$&     0.3514 & 0.3514\\ \hline             
				5  &   0.7174 +    0.4055i&$\C$&      0.2375& 0.2375   & 26   &   -1.3101 -    0.9518i &$\C$&     0.4286 & 0.4286 \\ \hline           
				6  &   0.6073 +    0.5570i&$\C$&      0.2375& 0.2375   & 27   &   0                    &$\R$&      0   &    0         \\ \hline             
				7  &   1.4638 +     1.0635i&$\C$&      0.4021 & 0.4021   & 28   &   -0.3421 -    0.7497i  &$\C$&    0.2375& 0.2375 
				\\ \hline            
				8  &   0.2589 +    0.7968i &$\C$&     0.3514 & 0.3514  &  29  &   -0.3435 -     1.0571i &$\C$&  3.3e-16 & 0.0045  \\ \hline           
				9  &   0.5004 +     1.5401i &$\C$&     0.4286 & 0.4286   & 30   &   -0.1639 -    0.8076i  &$\C$&    0.2375&  0.2375  \\ \hline          
				10  &   0.1236 +    0.3804i&$\C$&      0.4288  & 0.4288 & 31   &   -0.5591 -     1.7208i  &$\C$&    0.4021 & 0.4021   \\ \hline             
				11   &   -0.1639 +    0.8076i&$\C$&      0.2375  & 0.2375& 32   &   -0.1123 -    0.3457i &$\C$&     0.4151 & 0.4151 \\ \hline              
				12   &   -0.3435 +     1.0571i &$\C$&  3.3e-16 & 0.0045 &  33   &   0.2589 -    0.7968i &$\C$&     0.3514 &  0.3514  \\ \hline             
				13   &   -0.3421 +    0.7497i&$\C$&      0.2375 & 0.2375   & 34   &   0.5004 -     1.5401i  &$\C$&    0.4286 & 0.4286  \\ \hline            
				14   &   -0.5591 +     1.7208i &$\C$&     0.4021 & 0.4021  &  35   &   0.1236 -    0.3804i &$\C$&     0.4288 & 0.4288 \\ \hline              
				15   &   -0.1123 +    0.3457i&$\C$&      0.4151 &  0.4151 & 36   &   0.6073 -    0.5570i &$\C$&     0.2375 & 0.2375   \\ \hline             
				16   &   -0.6778 +    0.4924i &$\C$&     0.3514&  0.3514  &  37   &   0.8992 -    0.6533i &$\C$&  1.3e-16& 0.0040  \\ \hline             
				17   &   -1.3101 +    0.9518i &$\C$&     0.4286& 0.4286    & 38   &   0.2941 -    0.2137i &$\C$&     0.4151&  0.4151  \\ \hline           
				18   &   -0.3236 +    0.2351i&$\C$&      0.4288& 0.4288    & 39  &   0.7174 -    0.4055i &$\C$&     0.2375& 0.2375   \\ \hline         
				19   &   -0.3635               &$\R$&     0.4151 & 0.4151   &  40   &   1.4638 -     1.0635i  &$\C$&    0.4021 & 0.4021   \\ \hline           
				20   &   -1.1115              &$\R$&   1.2e-16 & 0.0040  &  41   &   0.4                &$\R$&    0.4288& 0.4288      \\ \hline             
				21   &   -0.8187 +   0.0937i &$\C$&     0.2375 & 0.2375  & &&&&	\\ \hline  
			\end{tabular} 
		\end{scriptsize}
	\end{table}  
\end{exm}

\begin{exm}
	\label{exm:14:d:qubit}\emph{\cite[Example 6.6]{AMM10}} Given tensor $\cS\in\rS^9\R^2$ with:
	$$ \cS_{1,1,1,1,1,1,1,2,2} = \frac{1}{6\sqrt{2}}; \cS_{1,1,2,2,2,2,2,2,2} = \frac{1}{6\sqrt{2}}.$$  
	The polynomial $zv(z)-u(z)$ has degree 55.  It has 46 roots, 8 of them are real and the other 38 are complex. One of the real roots $z=0$ has multiplicity 10.  The computational results are shown in Table \ref{table:exm:14:d:qubit}. 
	Clearly $|s_{9}| = |\cS_{2,2,2,2,2,2,2,2,2}|=0$. For the 46 roots $z$, (\ref{antifixqub1}) fails to hold for 14 roots, so we have $R= \{z_i,~i\in [46]\backslash I\}$, where
	$I= \{4,5,6,13,14,15,23,24,31,32,33,40,41,42\},$ and $R'= R\cap\R$. By using formulas (\ref{specnrmSform}), (\ref{specnrmSform1}), (\ref{realforspecnrm}), (\ref{realforspecnrm1}), we get: 
	\begin{equation*}
	\begin{aligned}
	\|\cS\|_{\sigma~~} &  = \max\{|s_9|, \lambda_q(z), z\in R \} = \max\{|s_9|, \lambda_v(z), z\in R_1 \} = 0.4127,\\
	\|\cS\|_{\sigma,\R} & = \max\{|s_9|,\lambda_q(z),z\in R'\} = \max\{|s_9|,\lambda_v(z),z\in R_1'\} =0.4127.
	\end{aligned}
	\end{equation*}

	\begin{table}\caption{Computational results for Example \ref{exm:14:d:qubit}.} \label{table:exm:14:d:qubit}
		\centering
		\begin{scriptsize}
			\begin{tabular}{|c|r|c|c|c||c|r|c|c|c|c|} \hline
				No. &  $z$\quad\quad\quad\quad   & $\F$ & $\lambda_q(z)$ & $\lambda_v(z)$ & No. &  $z$\quad\quad\quad\quad   & $\F$ & $\lambda_q(z)$ & $\lambda_v(z)$ \\ \hline 
				1&1   & $\R$&  0.375  &   0.375    &  24&-0.9956 -   0.0943i  &$\C$&    0.2207 & 0.2207 \\    \hline            
				2&1.6862   &$\R$&    0.4127 & 0.4127  &  25&-0.7784               &$\R$&    0.3499 &     0   \\    \hline            
				3&0.4087 +    0.2969i  &$\C$&    0.3765 & 0.3765   & 26&-1.9795     &$\R$&   0.3765  & 0.3765  \\    \hline          
				4&1.0394 +    0.7552i  &$\C$& 5.9e-17 & 0.0057&  27&-0.4798 -    0.3486i  &$\C$&    0.4127 & 0.4127 \\    \hline              
				5&0.8608 +     0.5089i  &$\C$&    0.2207 & 0.2207 & 28&-0.8090 -    0.5878i  &$\C$&      0.375 &  0.375  \\    \hline             
				6&0.7500 +    0.6614i    &$\C$&  0.2207 & 0.2207    & 29&-1.3642 -    0.9911i   &$\C$&   0.4127 & 0.4127  \\    \hline           
				7&0.6297 +    0.4575i  &$\C$&    0.3499 & 0.0001 &  30&-0.1561 -    0.4805i &$\C$&     0.3765 & 0.3765  \\    \hline             
				8&1.6014 +     1.1635i   &$\C$&   0.3765 & 0.3765    & 31&-0.3970 -     1.2219i  &$\C$& 3.3e-17  & 0.0054 \\    \hline          
				9&0.1833 +    0.5640i  &$\C$&    0.4127 & 0.4127   &  32&-0.3973 -    0.9177i &$\C$&     0.2207 & 0.2207  \\    \hline           
				10 & 0.3090 +    0.9511i  &$\C$&      0.375  & 0.375   &  33&-0.2180 -    0.9760i &$\C$&     0.2207 & 0.2207 \\    \hline         
				11& 0.5211 +     1.6037i  &$\C$&    0.4127  & 0.4127 &  34&-0.2405 -    0.7403i &$\C$&     0.3499 & 0.0001  \\    \hline            
				12&-0.1561 +    0.4805i &$\C$&     0.3765 & 0.3765 & 35&-0.6117 -     1.8826i  &$\C$&    0.3765 & 0.3765  \\    \hline              
				13&-0.3970 +     1.2219i   &$\C$& 3.3e-17  & 0.0054  &  36& 0.1833 -    0.5640i  &$\C$&    0.4127 &  0.4127 \\    \hline           
				14&-0.2180 +    0.9760i  &$\C$&    0.2207& 0.2207  &  37&0.3090 -    0.9511i &$\C$&       0.3750&   0.3750   \\    \hline            
				15&-0.3973 +    0.9177i   &$\C$&   0.2207  & 0.2207 &  38&0.5211 -     1.6037i &$\C$&     0.4127 &  0.4127 \\    \hline           
				16&-0.2405 +    0.7403i  &$\C$&    0.3499 & 0.0002&  39& 0.4087 -    0.2969i  &$\C$&    0.3765 & 0.3765     \\    \hline              
				17&-0.6117 +     1.8826i  &$\C$&    0.3765& 0.3765   & 40& 1.0394 -    0.7552i &$\C$&  5.9e-17& 0.0057     \\    \hline            
				18&-0.4798 +    0.3486i   &$\C$&   0.4127& 0.4127   & 41& 0.7500 -    0.6614i   &$\C$&   0.2207& 0.2207      \\    \hline           
				19&-0.8090 +    0.5878i  &$\C$&      0.375 & 0.375    & 42&0.8608 -     0.5089i &$\C$&     0.2207 & 0.2207   \\    \hline          
				20&-1.3642 +    0.9911i   &$\C$&   0.4127 & 0.4127    & 43& 0.6297 -    0.4575i &$\C$&     0.3499 & 0.0001  \\    \hline         
				21&-0.5052                &$\R$&    0.3765& 0.3765   &  44&1.6014 -     1.1635i  &$\C$&    0.3765 &  0.3765    \\    \hline          
				22&-1.2847    &$\R$& 4.4e-30  & 0.0002&  45& 0.5930  &$\R$&     0.4127  &  0.4127  \\    \hline             
				23&-0.9956 +   0.0943i  &$\C$&    0.2207 & 0.2207   & 46&0  &$\R$& 0 &    0    			\\    \hline          
			\end{tabular} 
		\end{scriptsize}
	\end{table} 
\end{exm}

\begin{exm}\emph{\cite[Example 6.7(b)]{AMM10}}
	\label{exm:15:2:d:qubit} Given tensor $\cS\in\rS^{10}\R^2$ with:
	$$ \cS_{1,1,1,1,1,1,1,1,2,2} = \frac{1}{3\sqrt{10}};  \cS_{1,1,2,2,2,2,2,2,2,2} = \frac{1}{3\sqrt{10}}.$$ 
	The polynomial $zv(z)-u(z)$ has degree 71.  It has 61 roots, 7 of them are real and the other 54 are complex. One of the real roots $z=0$ has multiplicity 11. For cleanness of the paper, we will not list all the roots. Clearly $|s_{10}| =0$.
	By using formulas (\ref{specnrmSform}), (\ref{specnrmSform1}), (\ref{realforspecnrm}), (\ref{realforspecnrm1}), we get: 
	\begin{equation*}
	\begin{aligned}
	\|\cS\|_{\sigma~~} & = \max\{|s_{10}|, \lambda_q(z), z\in R  \}  = \max\{|s_{10}|, \lambda_v(z), z\in R_1 \} = 0.3953,\\
	\|\cS\|_{\sigma,\R} &  = \max\{|s_{10}|,\lambda_q(z),z\in R'\} = \max\{|s_{10}|,\lambda_v(z),z\in R_1'\} =0.3953.
	\end{aligned}
	\end{equation*}
	According to \cite{AMM10}, $  \|\cS\|_{\sigma}=\sqrt{\frac{5}{32}}$.

\end{exm} 

\begin{exm} \label{exm:16:2:d:qubit} Given tensor $\cS\in\rS^{11}\R^2$ with:
	$$ \cS_{1,1,1,1,1,1,1,1,1,1,2} = \frac{\sqrt{7}}{5\sqrt{11}};  \cS_{1,1,1,1,1,1,2,2,2,2,2} = \frac{1}{5\sqrt{42}};  \cS_{1,2,2,2,2,2,2,2,2,2,2} =\frac{-\sqrt{7}}{5\sqrt{11}}.$$ 
	The polynomial $zv(z)-u(z)$ has degree 100. It has 100 roots, 8 of them are real and the other 92 are complex. For cleanness of the paper, we will not list all the roots. Clearly, $|s_{11}| =0$. By using formulas (\ref{specnrmSform}), (\ref{specnrmSform1}), (\ref{realforspecnrm}), (\ref{realforspecnrm1}), we get:  
	\begin{equation*}
	\begin{aligned}
	\|\cS\|_{\sigma~~} &  = \max\{|s_{11}|, \lambda_q(z), z\in R  \}  = \max\{|s_{11}|, \lambda_v(z), z\in R_1 \} =0.4125,\\
	\|\cS\|_{\sigma,\R} & = \max\{|s_{11}|,\lambda_q(z),z\in R'\} = \max\{|s_{11}|,\lambda_v(z),z\in R_1'\} =0.4125.
	\end{aligned}
	\end{equation*}
\end{exm}

\begin{exm}\emph{\cite[Example 6.9(b)]{AMM10}}
	\label{exm:17:d:qubit} Given tensor $\cS\in\rS^{12}\R^2$ with:
	$$ \cS_{1,1,1,1,1,1,1,1,1,1,1,2} = \frac{\sqrt{7}}{10\sqrt{3}} ; \cS_{1,1,1,1,1,1,2,2,2,2,2,2} =  \frac{-1}{10\sqrt{21}}  ; \cS_{1,2,2,2,2,2,2,2,2,2,2,2} = \frac{-\sqrt{7}}{10\sqrt{3}} .$$  
	The polynomial $zv(z)-u(z)$ has degree 121. It has 121 roots,  13 of them are real and the other 108 are complex. For cleanness of the paper, we will not list all the roots. Clearly
	$|s_{12}| =0$.  By using formulas (\ref{specnrmSform}), (\ref{specnrmSform1}), (\ref{realforspecnrm}), (\ref{realforspecnrm1}), we get:
	\begin{equation*}
	\begin{aligned}
	\|\cS\|_{\sigma~~} &  = \max\{|s_{12}|, \lambda_q(z), z\in R  \} = \max\{|s_{12}|, \lambda_v(z), z\in R_1 \} =0.3395,\\
	\|\cS\|_{\sigma,\R} & = \max\{|s_{12}|,\lambda_q(z),z\in R'\} = \max\{|s_{12}|,\lambda_v(z),z\in R_1'\} =0.3395.
	\end{aligned}
	\end{equation*}
	According to \cite{AMM10}, $\|\cS\|_{\sigma}=\sqrt{\frac{28}{243}}$.
	
\end{exm}

\bibliographystyle{plain}

\begin{thebibliography}{1}
\bibitem{AMM10} M. Aulbach, D. Markham and Mi. Murao, The maximally entangled symmetric state in terms of the geometric measure, \emph{New Journal of Physics}, 12 (2010), 073025.
\bibitem{Ban38} S.~Banach, ``\"Uber homogene Polynome in ($L^2$),'' \emph{Studia Math.}, \textbf{7} (1938), pp.~36--44.

\bibitem{BHSW06}
Daniel J. Bates, Jonathan D. Hauenstein, Andrew J Sommese,
and Charles W. Wampler.
\newblock Bertini: Software for Numerical Algebraic Geometry.
\newblock Available at bertini.nd.edu with permanent doi: dx.doi.org/10.7274/R0H41PB5.

\bibitem{BCSS98}  L. Blum, F. Cucker, M. Shub, and S. Smale, \emph{Complexity and Real Computation}, Springer Science \& Business Media, 2012.

\bibitem{CS} D. Cartwright, B. Sturmfels, The number of eigenvectors of a tensor,  \emph{Linear Algebra Appl.} 438 (2013), no. 2, 942-–952.

\bibitem{CHLZ12} B. Chen, S. He, Z. Li, and S. Zhang, Maximum block improvement and polynomial optimization,
\emph{ SIAM J. OPTIM.} 22 (2012),  87--107.

\bibitem{CXZ10} L. Chen, A. Xu and H. Zhu, Computation of the geometric measure of entanglement for pure multiqubit states,  \emph{Physical Review A} 82 (2010), 032301.

\bibitem{LMV00} L. de Lathauwer, B. de Moor, and J. Vandewalle, On the best rank-1 and
rank-$(R_1,R_2, . . . ,R_N)$ approximation of higher-order tensors, \emph{SIAM J. Matrix Anal.
Appl.} 21 (2000), pp. 1324--1342.

\bibitem{Derk} H.~Derksen, On the nuclear norm and the singular value decomposition of tensors,  \textit{Found.\ Comput.\ Math.}, \textbf{16} (2016), no.~3, pp.~779--811.

\bibitem{DFL} H.~Derksen, S.~Friedland, and L.-H.~Lim, Nuclear norm as a continuous measure of quantum entanglement and separability,  \textit{preprint}, (2016).

\bibitem{Dic}  R. H. Dicke, Coherence in Spontaneous Radiation Processes, \emph{Physical Review.} 93 (1) (1954), 99--110.

\bibitem{Fel58} W. Feller, \emph{An Introduction to Probability Theory and Its Applications}, vol I, J. Wiley, 1958.

\bibitem{Fri77}  S. Friedland, Inverse eigenvalue problems, {\it Linear Algebra Appl.} 17
 (1977), 15-51.

\bibitem{Fri13} S.~Friedland,  Best rank-one approximation of real symmetric tensors can be chosen symmetric, \textit{Front.\ Math.\ China}, \textbf{8} (2013), pp.~19--40.

\bibitem{FLccnm} S.~Friedland and L.-H.~Lim, Computational Complexity of Tensor Nuclear Norm,  arXiv:1410.6072v1.

\bibitem{FL1} S.~Friedland and L.-H.~Lim, The computational complexity of duality, \textit{preprint}, (2016), arxiv.org/abs/1601.07629

\bibitem{FLTN16}  S. Friedland and L.-H. Lim, Nuclear Norm of Higher Order Tensors,  arXiv:1410.6072v3.

\bibitem{FK16} S. Friedland and T. Kemp, On the entanglement of symmetric quantum states, 2016, \emph{in preparation}.

\bibitem{FMPS13} S. Friedland, V. Mehrmann, R. Pajarola and S.K. Suter, 
 On best rank one approximation of tensors, \emph{Numer. Linear Algebra Appl.} 20 (2013), 942--955.

\bibitem{FO14} S. Friedland and  G. Ottaviani, The number of singular vector tuples and uniqueness of best rank one approximation of tensors,
 \emph{Foundations of Computational Mathematics} 14, 6 (2014), 1209--1242.

\bibitem{FT15} S. Friedland and V. Tammali, Low-rank approximation of tensors,  \emph{Numerical Algebra, Matrix Theory, Differential-Algebraic Equations and 
Control Theory}, edited by P. Benner et all, Springer, 2015, 377-410, arXiv:1410.6089.

\bibitem{GFE09} D. Gross, S. T. Flammia, and J. Eisert, Most Quantum States Are Too Entangled To Be Useful As Computational Resources, Phys. Rev. Lett. 102, 190501, 2009.

\bibitem{HL13} C.~J. Hillar and L.-H. Lim, Most tensor problems are NP-hard, \textit{J.\ ACM}, \textbf{60} (2013), no.~6, Art.~45, 39 pp.

\bibitem{HS14} J, Hubbard and D. Schleicher, Multicorns are not path connected, in \emph{Frontiers in Complex Dynamics}, p. 73--101,edited by 
A. Bonifant, M. Lyubich,  S. Sutherland, Princeton University Press, 2014.

\bibitem{Hubetall09} R. H\"ubener, M. Kleinmann, T.-C. Wei, C. Gonz\'alez-Guill\'en, and O. G\"uhne, \emph{Phys. Rev. A} 80 (2009), 032324. 

\bibitem{Jungetall08} E. Jung, M.-R. Hwang, H. Kim, M.-S. Kim, D. Park, J.-W. Son, and S. Tamaryan, Reduced state uniquely defines the Groverian measure of the original pure state, \emph{Phys. Rev. A} 77, 062317 (2008).

\bibitem{MGBB10}  J. Martin, O. Giraud, P.A. Braun, D. Braun and T. Bastin,  arXiv:1003.0593.

\bibitem{Mil} J. Milnor, \emph{Singular Points of Complex Hypersurfaces}, Princeton University Press, 1968.

\bibitem{NR96} C.A.  Neff and J.H. Reif, J. H. (1996). An efficient algorithm for the complex roots problem. \emph{J. Complexity}, 12 (1996), 81--115.

\bibitem{Nie14} J. Nie and L. Wang, Semidefinite relaxations for best rank-1 tensor approximations. \textit{SIAM Journal on Matrix Analysis and Applications}, \textbf{35} (2014), no.~3, pp.~ 1155--1179. 

\bibitem{Ren05} R. Renner, Security of Quantum Key Distribution, arXiv:quant-ph/0512258

\bibitem{Sma97} S. Smale, Complexity theory and numerical analysis,  \emph{Acta Numer.} 6 (1997), 523--551.

\bibitem{Syl51} J. J. Sylvester, On a remarkable discovery in the theory of canonical forms and of hyperdeterminants, \emph{Phil.
Mag.} 2 (1851), 391--410 ( Collected papers, Vol. I, Paper 41). 

\bibitem{TWP09} S.Tamaryan, T.C. Wei and D. Park, Maximally entangled three-qubit states via geometric measure of entanglement, \emph{Physical Review A}. 80(5) (2009), 052315.
 

\bibitem{Tong} T.~Zhang and G.~H.~Golub, Rank-one approximation to high order tensors, \textit{SIAM J.\ Matrix Anal.\ Appl.}, \textbf{23} (2001), no.~2, pp.~534--550.




 \end{thebibliography}

\end{document}